\documentclass[12pt]{article}
\usepackage[latin1]{inputenc}

\usepackage{amsmath}
\usepackage{amsfonts}
\usepackage[english]{babel}
\usepackage{epsfig}
\usepackage{framed}
\usepackage{graphicx}
\usepackage[latin1]{inputenc}
\usepackage{latexsym}
\usepackage{psfrag}
\usepackage{subfigure}
\usepackage{theorem}
\newlength{\blank}
\theoremheaderfont{\scshape}
\theorembodyfont{\normalfont\slshape}
\newtheorem{thm}{Theorem}[section]
\newtheorem{prb}[thm]{Unsolved Problem}
\newtheorem{prp}[thm]{Proposition}
\newtheorem{cnj}[thm]{Conjecture}

\newtheorem{cor}[thm]{Corollary}
\newtheorem{lem}[thm]{Lemma}

\newtheorem{cas}{Case}

\newtheorem{sta}{\hspace{-\blank}}[thm] 

\newcommand{\qed}{\hfill $\Box$}
\newenvironment{proof}{\noindent\underline{Proof}:}{\hfill\qed\par\medskip}

\def\zz{\mathbb{Z}}

\newcommand{\ignore}[1]{}

\newcommand{\rmenum}
{
\renewcommand{\theenumi}{{\rm(\roman{enumi})}}
\renewcommand{\labelenumi}{\theenumi}
}

\newcommand{\Contra}[3]{#1/(#2\rightarrow#3)}

\rmenum

\begin{document}
  \title{On Two Unsolved Problems\\Concerning\\Matching Covered
    Graphs\\
    {\large{Dedicated to the memory of Professor W.T.Tutte}}\\
    {\large{on the occasion of the centennial of his birth}}}
  \author{
    Cl{\'a}udio L.~Lucchesi
    \and
    Marcelo~H.~de~Carvalho
    \and
    Nishad~Kothari
    \and
    U.~S.~R.~Murty
  }
  \date{26 May 2017}

  \maketitle
  \thispagestyle{empty}

  \begin{abstract}
    \noindent
    A cut $C:=\partial (X)$ of a  matching covered graph $G$ is a {\em
      separating  cut}   if  both   its  $C$-contractions   $G/X$  and
    $G/\overline{X}$ are also matching covered. A brick is {\em solid}
    if it is free of nontrivial separating cuts.
    We (Carvalho, Lucchesi and Murty) showed in \cite{clm04}
    that the  perfect matching  polytope of a  brick may  be described
    without  recourse to odd set  constraints if  and only  if it  is
    solid, and in  \cite{clm06} we proved that the  only simple planar
    solid bricks  are the  odd wheels.  The problem  of characterizing
    nonplanar solid bricks remains unsolved.

    \smallskip

    A {\em bi-subdivision} of a graph $J$ is a graph obtained from $J$
    by replacing each of its edges by paths of odd length.  A matching covered graph
    $J$ is a  {\em conformal minor} of a matching covered graph $G$  if there
    exists a bi-subdivision $H$ of $J$ which is a subgraph of $G$ such
    that $G-V(H)$ has a perfect matching. For a fixed matching covered
    graph $J$, a matching covered graph  $G$ is {\em $J$-based} if $J$
    is  a  conformal   minor  of  $G$  and,  otherwise,   $G$  is  {\em $J$-free}. 
    A  basic result due to  Lov\'asz~\cite{lova83} states
    that  every   nonbipartite  matching   covered  graph   is  either
    $K_4$-based   or   is   $\overline{C_6}$-based  or   both,   where
    $\overline{C_6}$ is the triangular  prism. In \cite{komu16}, we (Kothari and Murty)
    showed that, for any  cubic brick $J$,  a  matching covered
    graph $G$ is $J$-free if and only if each of its bricks is $J$-free.
    We also  found characterizations  of planar
    bricks    which   are    $K_4$-free    and    those   which    are
    $\overline{C_6}$-free. Each of these problems remains unsolved  in the
    nonplanar case.
    
    \smallskip

    In this  paper we  show that the  seemingly unrelated  problems of
    characterizing nonplanar solid bricks on  the one hand, and on the
    other of characterizing nonplanar $\overline{C_6}$-free bricks are
    essentially the same. We do this by establishing that a simple nonplanar brick,
    other  than the  Petersen graph,  is  solid if  and only  if it  is
    $\overline{C_6}$-free. 
    In order to prove  this, we first show that
    any nonsolid  brick has one  of the four  graphs $\overline{C_6}$,
    the  bicorn,  the tricorn  and  the  Petersen graph  (depicted  in
    Figure~\ref{fig:four-bricks}) as a conformal minor.  Then, using a
    powerful theorem  due to Norine and  Thomas~\cite{noth07}, we show
    that the bicorn, the tricorn  and the Petersen graph are dead-ends
    in  the  sense that  any  simple  nonplanar nonsolid  brick  which
    contains any one  of these three graphs as a
    proper conformal minor  also  contains $\overline{C_6}$ as a conformal minor.

  \end{abstract}
  \tableofcontents

  \section{Background and Preliminaries}

  \subsection{Matching Covered Graphs}
  \label{sec:mcg}
  For  graph  theoretical  terminology and  notation,  we  essentially
  follow  the  book  by  Bondy  and  Murty~\cite{bomu08}.  All  graphs
  considered in this paper are loopless.

  \smallskip

  For a  subset $S$ of the  vertex set of  a graph $G$, we  denote the
  number of  odd components of $G-S$  by $o(G-S)$. Tutte~\cite{tutt47}
  established the following fundamental theorem.

  \begin{thm}[{\sc Tutte's Theorem}]\label{thm:tutt}
    A  graph $G$  has  a perfect  matching if  and  only if  $$o(G-S) \leq |S|$$ for any subset $S$ of $V(G)$.
  \end{thm}

  \smallskip

  An edge $e$
  of a graph $G$ is {\em  admissible} if there is some perfect matching
  of  $G$ that  contains  it.  A {\em  matching  covered}  graph is  a
  connected  graph of  order  at  least two  in  which  every edge  is
  admissible. A  simple argument shows  that a matching  covered graph
  cannot have a cut vertex. 
  Tutte~\cite{tutt47}  used  Theorem~\ref{thm:tutt}  to  strengthen  a
  classical  theorem  of J.  Petersen  (1891)  by showing  that  every
  $2$-connected cubic  graph is  matching covered.  

Let $G$ be a graph that has a perfect matching. A subset $B$ of $V(G)$
is a {\it barrier} if $o(G-B) = |B|$. Using Tutte's Theorem, one may easily deduce the following
characterization of inadmissible edges:

\begin{prp}
\label{prp:inadmissible}
Let $G$ be a graph that has a perfect matching. An edge $e$ is inadmissible
if and only if there exists a barrier that contains both ends of $e$.
\end{prp}

This yields the following characterization of matching covered graphs:

\begin{cor}
Let $G$ be a connected graph that has a perfect matching.
Then $G$ is matching covered if and only if every barrier of $G$ is a stable set.
\end{cor}

  There  is an  extensive theory  of matching  covered graphs  and its
  applications.  In the  book by  Lov\'asz and  Plummer~\cite{lopl86},
  matching  covered  graphs are  referred  to  as {\em  $1$-extendable
    graphs}. The terminology we use here was  introduced by Lov\'asz
  in his seminal work~\cite{lova87} and in the follow-up work by three of
  us  in~\cite{clm02}.  This  work  relies  on  a  number  of  notions
  introduced and results proved by us and, among others, Lov\'asz, and
  Norine and Thomas. For the benefit of the readers, we shall describe
  them  and  provide references.  For  uniformity,  we have  found  it
  necessary, in  some cases,  to change  the notation  and terminology
  used in the original sources.

  \smallskip

  A number  of cubic
  graphs play special roles in  this theory. They include the complete
  graph     $K_4$,     and     the    four     graphs     shown     in
  Figure~\ref{fig:four-bricks}, namely  the triangular prism  which is
  denoted  by $\overline{C_6}$  because it  is the  complement of  the
  $6$-cycle,  the bicorn  and the  tricorn (as  they resemble,  in our
  imagination,  the  two-cornered  and  three-cornered  hats  worn  by
  pirates),  and the  ubiquitous  Petersen graph  which  we denote  by
  $\mathbb{P}$.

  \bigskip
  \begin{figure}[!ht]
    \begin{center}
      \psfrag{a}{$(a)$}
      \psfrag{b}{$(b)$}
      \psfrag{c}{$(c)$}
      \psfrag{d}{($d$)}
      \psfrag{v1}{$v_1$} 
      \psfrag{v2}{$v_2$} 
      \psfrag{v3}{$v_3$}
      \psfrag{v4}{$v_4$} 
      \psfrag{v5}{$v_5$} 
      \psfrag{v6}{$v_6$}
      \psfrag{v7}{$v_7$} 
      \psfrag{v8}{$v_8$} 
      \psfrag{v9}{$v_9$}
      \psfrag{v10}{$v_{10}$}
      \epsfig{file=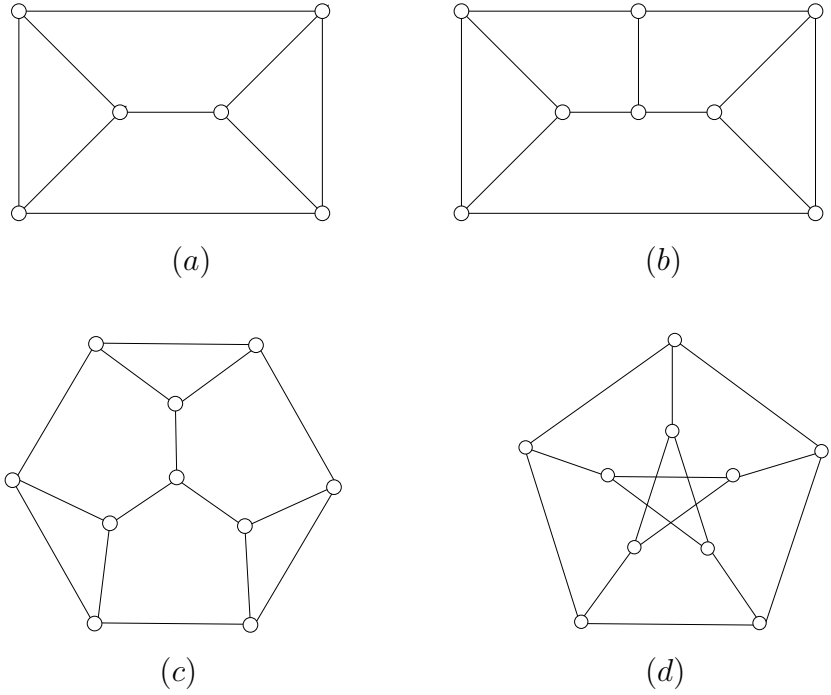, width=.8\textwidth}
    \end{center}
    \caption{(a) $\overline{C_6}$,  (b) the  bicorn, (c)  the tricorn,
      and (d) $\mathbb{P}$}
    \label{fig:four-bricks}
  \end{figure}

  \subsection{Bi-subdivisions}
  A {\em  bi-subdivision} of an  edge $e$ of  a graph $J$  consists of
  subdividing  it by  inserting an  even number  of vertices.
  A graph $H$ obtained from $J$ by bi-subdividing each edge, in any subset of the edges,
 is called a  {\em bi-subdivision}  of  $J$.  (The  term
  `bi-subdivision' is due to McCuaig~\cite{mccu01}. The same notion
  has been  called an `even subdivision'  by some authors, and  `odd
  subdivision' by  some others). If  $J$ is a matching  covered graph,
  then  any bi-subdivision  $H$ of  $J$ is  also matching  covered; in
  fact,  there is  a  one-to-one correspondence  between  the sets  of
  perfect matchings  of $J$ and of $H$.  Figure~\ref{fig:bi-subdivision}
  shows bi-subdivisions of the complete graph $K_4$ and of the triangular
  prism $\overline{C_6}$.

  \begin{figure}[!ht]
    \psfrag{a}{$(a)$} \psfrag{b}{$(b)$}
    \begin{center}
      \epsfig{file=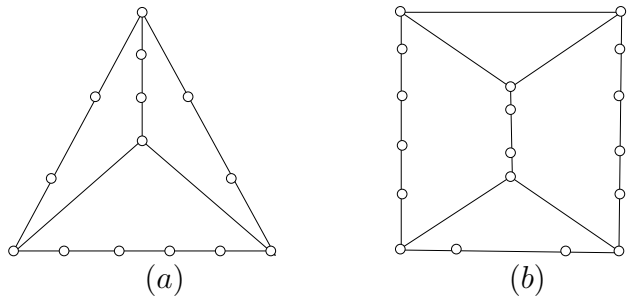, width=.6\textwidth}
      \caption{(a) A bi-subdivision of  $K_4$; (b) a bi-subdivision of
	$\overline{C_6}$}
      \label{fig:bi-subdivision}
    \end{center}
  \end{figure}


  \subsection{Splicing and Separation}

  \subsubsection{The operation of splicing}
  Let $G_1$  with a specified vertex  $u$, and $G_2$ with  a specified
  vertex $v$, be  two disjoint graphs. Suppose that the  degree of $u$
  in $G_1$ and the degree of $v$ in $G_2$ are the same, and that $\pi$
  is a bijection  between the set $\partial_{1}(u)$ of  edges of $G_1$
  incident with $u$,  and the set $\partial_{2}(v)$ of  edges of $G_2$
  incident with $v$. We denote by $(G_1\odot G_2)_{u,v,\pi}$ the graph
  obtained from the union of $G_1-u$  and $G_2-v$ by joining, for edge
  $e$ in  $\partial_1(u)$, the  end of  $e$ in $G_1-u$  to the  end of
  $\pi(e)$ in $G_2-v$,  and refer to it as the  graph obtained by {\em
    splicing} $G_1$ {\em at} $u$ with  $G_2$ {\em at} $v$ with respect
  to  the bijection  $\pi$.   The proof  of  following proposition  is
  straightforward:

  \begin{prp}\label{prp:splicing}
    The  graph $(G_1\odot  G_2)_{u,v,\pi}$  obtained  by splicing  two
    matching  covered   graphs  $G_1$  and  $G_2$   is  also  matching
    covered.\qed
  \end{prp}

  In  general, the  result  of  splicing two  graphs  $G_1$ and  $G_2$
  depends  on the  choices of  $u$, $v$,  $\pi$. (Both  the pentagonal
  prism and  the Petersen graph  can be  realized as splicings  of two
  copies  of the  $5$-wheel  at their  hubs.)  However,  if  $H$ is  a
  vertex-transitive cubic graph, then the result of splicing $G_1=K_4$
  with $G_2=H$ does  not depend, up to isomorphism, on  the choices of
  $u$, $v$, and  $\pi$, and we denote it simply  by $K_4\odot H$. More
  generally, for any cubic graph $H$, the result of splicing $K_4$ and
  $H$  depends,  up   to  isomorphism,  only  on  the   orbit  of  the
  automorphism group of  $H$ to which $v$ belongs (and  the choices of
  $u$ and $\pi$ are immaterial); and we denote it simply by $(K_4\odot
  H)_v$.

  \smallskip

  For   example,   since   both   $K_4$   and   $\overline{C_6}$   are
  vertex-transitive,  there is  only one  way of splicing $K_4$  with
  itself or with $\overline{C_6}$. Thus $K_4\odot K_4=\overline{C_6}$,
  and $K_4\odot\overline{C_6}$  is the  bicorn.  But  the automorphism
  group  of  the bicorn  has  three  orbits and,  consequently,  three
  different graphs  (one of which is  the tricorn) can be  produced by
  splicing  $K_4$   with  the  bicorn   (Figure~\ref{fig:figR8}).  The
  automorphism group of the tricorn also has three orbits and splicing
  $K_4$   with    the   tricorn   yields   three    different   graphs
  (Figure~\ref{fig:figR10}). 

  \begin{figure}[!ht]
    \begin{center}
      \psfrag{a}{($a$)} 
      \psfrag{b}{($b$)} 
      \psfrag{c}{($c$)}
      \epsfig{file=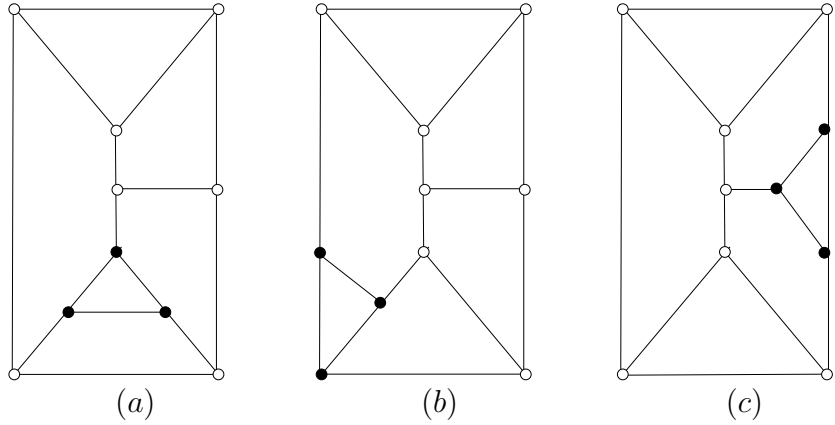, width=.8\textwidth}
      \caption{Cases of splicing $K_4$ and the bicorn}
      \label{fig:figR8}
    \end{center}
  \end{figure}

  \begin{figure}[!ht]
    \begin{center}
      \epsfig{file=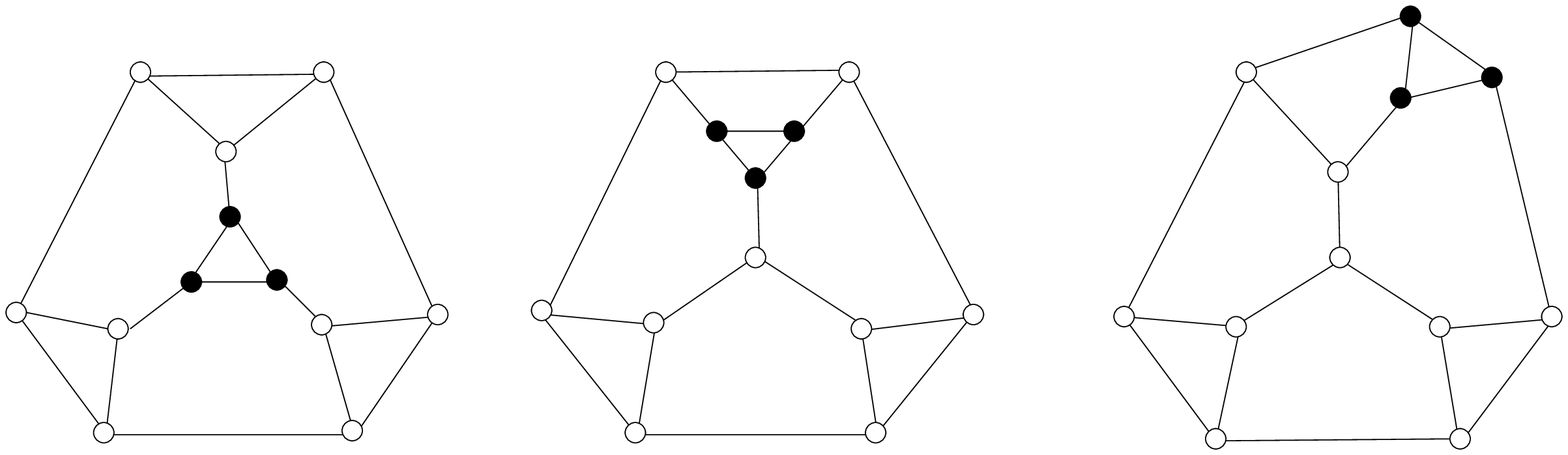, width=\textwidth}
      \caption{Cases of splicing $K_4$ and the tricorn}
      \label{fig:figR10}
    \end{center}
  \end{figure}

%
%

  \subsubsection{Cuts and cut-contractions}
  For a subset $X$ of the vertex  set $V(G)$ of a graph $G$, we denote
  the  set of  edges of  $G$ which  have exactly  one end  in $X$  by
  $\partial(X)$  and  refer  to  it  as  the  {\em  cut} of $X$.
(For a vertex $v$ of $G$, we simplify    the   notation    $\partial(\{v\})$   to
  $\partial(v)$.)        If      $G$       is      connected       and
  $C:=\partial(X)=\partial(Y)$,         then          $Y=X$         or
  $Y=\overline{X}=V -   X$,   and   we   refer   to   $X$   and
  $\overline{X}$ as the {\em shores} of~$C$.

  \smallskip

  For a cut $C$ of a matching covered graph, the parities of the cardinalities
of the two shores are the same. Here, we shall only be concerned with those cuts
that have shores of odd cardinality.
A cut is {\em trivial} if  either shore has just one
  vertex, and is  {\em nontrivial} otherwise.

  \smallskip

  Given any  cut $C:=\partial(X)$  of a  graph $G$,  one can  obtain a
  graph by  shrinking $X$ to a  single vertex ${x}$ (and  deleting any
  resulting loops);  we denote it by  $\Contra{G}{{X}}{{x}}$ and refer
  to the vertex ${x}$ as its  {\em contraction vertex}. The two graphs
  $\Contra{G}{X}{x}$ and  $\Contra{G}{\overline{X}}{\overline{x}}$ are
  the  two {\em  $C$-contractions}  of  $G$.  When  the  names of  the
  contraction  vertices  are  irrelevant   we  shall  denote  the  two
  $C$-contractions of $G$ simply by $G/{X}$ and $G/\overline{X}$.

  \subsubsection{Separating cuts}

  A cut  $C:=\partial(X)$ of  a matching  covered graph  $G$ is  {\em
    separating} if both the  $C$-contractions of  $G$ are
  also matching covered. All trivial cuts are clearly separating cuts.
  Figures~\ref{fig:sep-cuts}(a)  and (b)  show examples  of separating
  cuts, but the cut indicated in Figure~\ref{fig:sep-cuts}(c) is not a
  separating cut.

  \begin{figure}[!ht]
    \psfrag{a}{$(a)$}
    \psfrag{b}{$(b)$}
    \psfrag{c}{$(c)$}
    \begin{center}
      \epsfig{file=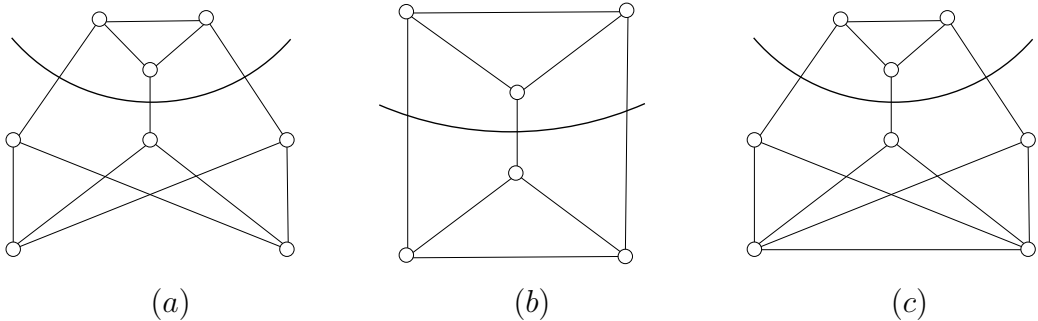, width=\textwidth}
      \caption{(a) and (b) are separating cuts, but (c) is a cut that is
	not}
      \label{fig:sep-cuts}
    \end{center}
  \end{figure}

  The   following  proposition   provides a necessary   and  sufficient
  condition  under which  a  cut in  a matching  covered  graph is  a
  separating cut, and is easily proved.

  \begin{prp}[\protect{\cite[Lemma 2.19]{clm02}}]
    \label{prp:sep-cut}
    A cut $C$ of  a matching covered graph $G$ is  a separating cut if
    and only if, given any edge $e$, there is a perfect matching $M_e$
    of $G$ such that $e \in M_e$ and $|C\cap M_e|=1$.\qed
  \end{prp}

  Let $G_1$ and  $G_2$ be two disjoint matching  covered graphs. Then,
  as noted before, any  graph $G=(G_1\odot G_2)_{u,v,\pi}$ obtained by
  splicing $G_1$ and $G_2$ is  also matching covered.  Clearly the cut
  $C:=\partial (V(G_1)-u)=\partial  (V(G_2)-v)$, which we refer  to as
  the {\em  splicing cut}, is a  separating cut of $G$,  and $G_1$ and
  $G_2$  are   the  two  $C$-contractions  of   $G$.   Conversely,  if
  $C:=\partial(X)$ is  a separating  cut of  a matching  covered graph
  $G$,  then  $G$  can  be recovered  from  its  two  $C$-contractions
  $G_1:=\Contra{G}{\overline{X}}{\overline{x}}$                    and
  $G_2:=\Contra{G}{X}{x}$,  by   splicing  them  at   the  contraction
  vertices   with   respect   to    the   identity   mapping   between
  $\partial_1(\overline{x})$,   which    is   equal   to    $C$,   and
  $\partial_2(x)$,  which  is also  equal  to  $C$. Thus,  a  matching
  covered graph $G$ has a nontrivial  separating cut if and only if it
  can  be obtained  by splicing  two smaller  matching covered  graphs
  $G_1$ and $G_2$.

  \subsubsection{Separating cut decompositions}
  Suppose  that $G$  is a  matching  covered graph  with a  nontrivial
  separating cut $C$.  Then the two $C$-contractions of  $G$ provide a
  decomposition of  $G$ into two  smaller matching covered  graphs. If
  either  $G_1$ or  $G_2$ has  nontrivial separating  cuts, then  that
  graph  can   be  decomposed  into  even   smaller  matching  covered
  graphs. By applying this  procedure repeatedly, any matching covered
  graph may be decomposed into a list of matching covered graphs which
  are free  of nontrivial separating  cuts. However, depending  on the
  choice of cuts  used in the decomposition procedure,  the results of
  two `separating cut decompositions' may result in entirely different
  lists  of graphs  which are  free of  separating cuts.  For example,
  consider     the     matching     covered     graph     shown     in
  Figure~\ref{fig:2-sep-cut-decomps}    with   the    four   indicated
  separating cuts $C_1$, $C_2$, $C_3$, and $C_4$. By first considering
  cut-contractions  with respect  to $C_1$,  we obtain  a $K_4$  (with
  multiple edges) which  is free of nontrivial separating  cuts, and a
  graph in  which $C_2$  is a  nontrivial separating  cut. Of  the two
  $C_2$-contractions  of  this second  graph,  one  is a  $K_4$  (with
  multiple  edges) and  the  other is  a  graph in  which  $C_3$ is  a
  nontrivial  separating cut.  The  $C_3$-contractions  of this  third
  graph  are two  copies of  $K_4$ (with  multiple edges).  Thus, this
  sequence of separating cut contractions yields a list of four copies
  of  $K_4$  (with  multiple  edges).
However, just  one application of the decomposition
  procedure which  involves the  cut $C_4$, yields  two copies  of the
  $5$-wheel  (with  multiple  edges)  which   happen  to  be  free  of
  nontrivial separating cuts.

  \begin{figure}[!ht]
    \psfrag{C1}{$C_1$}
    \psfrag{C2}{$C_2$}
    \psfrag{C3}{$C_3$}
    \psfrag{C4}{$C_4$}
    \begin{center}
      \epsfig{file=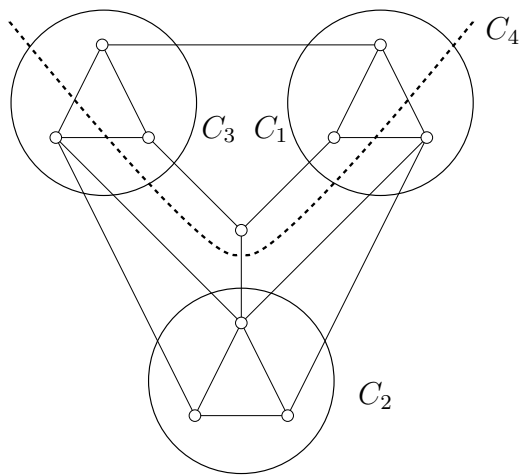, width=.5\textwidth}
      \caption{A  matching covered  graph  that  admits two  different
	separating cut decompositions}
      \label{fig:2-sep-cut-decomps}
    \end{center}
  \end{figure}

  \subsection{Bricks and Braces}

  \subsubsection{Tight cuts, bricks and braces}
  A cut $C$  in a matching covered  graph $G$ is a {\em  tight cut} of
  $G$  if $|C\cap  M|=1$ for  every perfect  matching $M$  of $G$.  It
  follows from  Proposition~\ref{prp:sep-cut} that every tight  cut of
  $G$ is also a separating cut  of $G$. However, the converse does not
  always     hold.    For     example,     the     cut    shown     in
  Figure~\ref{fig:sep-cuts}(b) is  a separating cut,  but it is  not a
  tight cut.

  A matching covered graph, which is free of nontrivial tight cuts, is a
  {\em  brace} if  it is  bipartite  and is a {\em brick} if it is nonbipartite.

  \subsubsection{Tight cut decompositions}
  \label{sec:bricks-braces}
    A {\em tight cut decomposition} of a matching covered
  graph  consists  of  applying the previously described separating  cut
  decomposition   procedure    where   we   restrict    ourselves   to
  cut-contractions  with   respect  to  nontrivial  tight   cuts. Clearly, any
  application  of the tight cut decomposition procedure  on a  given matching  covered
  graph produces a  list of bricks  and braces.   In striking
  contrast to  the separating  cut decomposition procedure,  the tight
  cut decomposition  procedure has the following  significant property
  established by Lov\'asz~\cite{lova87}.

  \begin{thm}[Uniqueness of the Tight Cut Decomposition]
    ~\\Any two  applications of the tight  cut decomposition procedure
    on a  match\-ing covered graph yield  the same list of  bricks and
    braces (up to multiple edges).
  \end{thm}

  In particular, any  two applications of the  tight cut decomposition
  procedure on a  matching covered graph $G$ yield the  same number of
  bricks; we denote  this invariant by $b(G)$ and refer  to it as the
  {\em number  of  bricks}  of~$G$.
  
  \subsubsection{Barrier cuts and $2$-separation cuts}
  Let $G$ be a matching covered graph.
  If $B$ is a barrier
  of  $G$ then,  for  any perfect  matching  $M$ of  $G$  and any  odd
  component  $K$  of $G-B$,  a  simple  counting argument  shows  that
  $|M\cap\partial(V(K))|=1$  (and   also  that   $G-B$  has   no  even
  components).  Consequently,  $\partial(V(K))$ is a tight  cut of $G$
  for any component $K$ of  $G-B$. Tight cuts  of $G$ which  arise in
  this manner are called {\em barrier cuts associated with the barrier
    $B$} (see Figure~\ref{fig:sep-cuts}(a)).

  \smallskip

  We shall  refer to  a vertex  cut $\{u,v\}$  of $G$  which is  not a
  barrier  as a  \mbox{\em $2$-separation}  of  $G$.  When  $\{u,v\}$ is  a
  $2$-separation of $G$, the fact that $\{u,v\}$ is not a barrier
implies  that  each  component  of  the
  disconnected graph $G-u-v$  is even. Let $S$ denote  the vertex set
  of  the union  of  a nonempty  proper subset  of  the components  of
  $G-u-v$.      It    can     then    be     verified    that     both
  $C:=\partial(S\cup\{u\})$  and  $D:=\partial(S\cup\{v\})$ are  tight
  cuts of $G$.  Tight cuts which  arise in this manner are called {\em
    $2$-separation cuts}. See Figure~\ref{fig:2-sep-cut}.

    \begin{figure}[h!t]
    \begin{center}
      \psfrag{u}{$u$} 
      \psfrag{v}{$v$} 
      \psfrag{C}{$C$} 
      \psfrag{D}{$D$}
      \psfrag{H1}{} 
      \psfrag{H2}{}
      \psfrag{L}{}
      \psfrag{L'}{}
      \epsfig{file=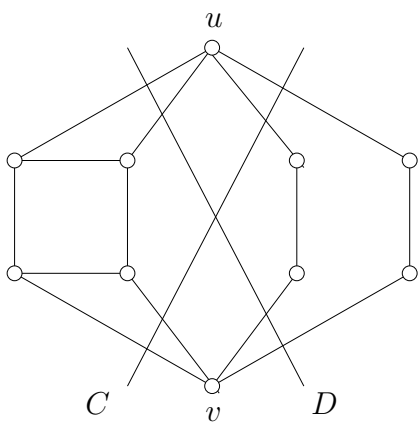,width=.4\textwidth}
      \caption{Two $2$-separation cuts in a matching covered graph} 
      \label{fig:2-sep-cut}
    \end{center}
  \end{figure}

%

  \smallskip

  An {\em ELP-cut} in a matching covered graph is a tight cut which is
  either a barrier  cut or is a $2$-separation cut.   A theorem due to
  Edmonds,  Lov\'asz and  Pulleyblank~\cite{elp82}  states  that if  a
  matching covered  graph has  nontrivial tight cuts,  then it  has an
  ELP-cut. The  following characterization of bricks  is a consequence
  of that basic result.

  \begin{thm}[{\sc The ELP Theorem}]\label{thm:elp}
    A  matching  covered  graph is  a  brick  if  and  only if  it  is
    $3$-connected and is free of nontrivial barriers.
  \end{thm}

  Characterization  of  braces  can  be  found  in  \cite{lova87} and \cite{lopl86}.

%

  \subsubsection{Six families of bricks and braces}

  We  now describe  six  families  of graphs  that  are of  particular
  interest in this work.

  \subsubsection*{Odd Wheels}
  Let $C_k$ be an  odd cycle of length at least  three. Then, the {\em
    odd wheel} $W_k$ is defined to be the join of $C_k$ and $K_1$. The
  smallest odd wheel is $W_3\cong K_4$.   For $k\geq 5$, $W_k$ has one
  vertex  of degree  $k$,  called  its {\em  hub};  the remaining  $k$
  vertices lie on a cycle which is referred to as the {\em rim}. Every
  odd wheel is a brick.

  \subsubsection*{Biwheels}
  Let $C_{2k}$ be an even cycle of length six or more with bipartition
  $(X,X')$, and let  $h$ and $h'$ be two vertices  ({\em hubs}) not on
  that cycle.   The graph obtained  by joining  $h$ to each  vertex in
  $X$, and $h'$  to each vertex in  $X'$, is known as  a {\em biwheel}
  with  $h$ and  $h'$ as  its hubs.  We shall  denote it  by $B_{2k}$.
  Figure~\ref{fig:biwheels}(a) shows a biwheel  on eight vertices (the
  two half-edges  labelled $e$  are to be  identified to  complete the
  rim); it is isomorphic to the cube.  Every biwheel is a brace.

  \subsubsection*{Truncated biwheels}
  Let $(v_1,v_2,\dots,v_{2k})$ be  a path of odd  length, where $k\geq
  2$, and let  $h$ and $h'$ be  two vertices ({\em hubs})  not on that
  path.   We shall  refer  to the  graph obtained  by  joining $h$  to
  vertices in  $\{v_1,v_3,\dots,v_{2k-1}\}\cup\{v_{2k}\}$, and joining
  $h'$ to vertices in  $\{v_1\}\cup\{v_2,v_4,\dots,v_{2k}\}$ as a {\em
    truncated biwheel}. We shall denote  it by $T_{2k}$.  The smallest
  truncated          biwheel         is          isomorphic         to
  $\overline{C_6}$.  Figure~\ref{fig:biwheels}(b)  shows  a  truncated
  biwheel on eight vertices.  Every truncated biwheel is a brick.

  \begin{figure}[!ht]
    \begin{center}
      \psfrag{a}{$(a)$}
      \psfrag{b}{$(b)$}
      \psfrag{e}{$e$}
      \epsfig{file=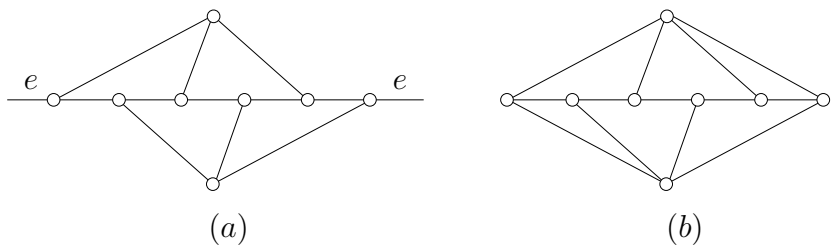,width=.8\textwidth}
      \caption{(a) A biwheel, (b) a truncated
	biwheel.}
      \label{fig:biwheels}
    \end{center}
  \end{figure}
  Norine and Thomas~\cite{noth07} refer to truncated biwheels as lower
  prismoids.

  \subsubsection*{Prisms}
  Let $(u_1,u_2,\dots,u_k,u_1)$ and  $(v_1,v_2,\dots, v_k,v_1)$ be two
  disjoint cycles of length at least three. The graph on $2k$ vertices
  obtained from  the union  of these  two cycles  by joining  $u_i$ to
  $v_i$, for $1\leq i\leq k$ is  the {\em $k$-prism}. (In other words,
  the $k$-prism  is the Cartesian  product of the $k$-cycle  $C_k$ and
  the  complete  graph  $K_2$.)   We shall  denote  the  $k$-prism  by
  $P_{2k}$.   The   $3$-prism  $P_6$,  commonly  known   as  the  {\em
    triangular  prism}, is  isomorphic  to  $\overline{C_6}$, and  the
  $4$-prism  is   isomorphic  to  the   cube.   The  graph   shown  in
  Figure~\ref{fig:bricks-indices}(b) is the  $5$-prism, commonly known
  as the {\em pentagonal prism}.  (For every odd $k$, the $k$-prism is
  a brick and for every even $k$, the $k$-prism is a brace.)

  \subsubsection*{M{\"o}bius Ladders}
  Let $(u_1,u_2,\dots,u_k)$ and $(v_1,v_2,\dots, v_k)$ be two disjoint
  paths of length at least two.   The graph obtained from the union of
  these two paths by joining $u_i$ to $v_i$, for $1\leq i\leq k$, and,
  in addition, joining $u_1$ to $v_k$, and $u_k$ to $v_1$, is known as
  the {\em M\"{o}bius ladder} of  order $2k$. The M\"{o}bius ladder of
  order six is  isomorphic to $K_{3,3}$, and the  M\"{o}bius ladder of
  order eight is shown in Figure~\ref{fig:ladders}(a) (this drawing is
  to be taken  as an embedding on the M\"{o}bius  strip).  When $k$ is
  odd, the M\"{o}bius ladder of order $2k$ is a brace and, when $k$ is
  even, the M\"{o}bius ladder of order $2k$ is a brick.
  
  \begin{figure}[!ht]
    \begin{center}
      \psfrag{a}{$(a)$}
      \psfrag{b}{$(b)$}
      \psfrag{e}{$e$}
      \psfrag{f}{$f$}
      \epsfig{file=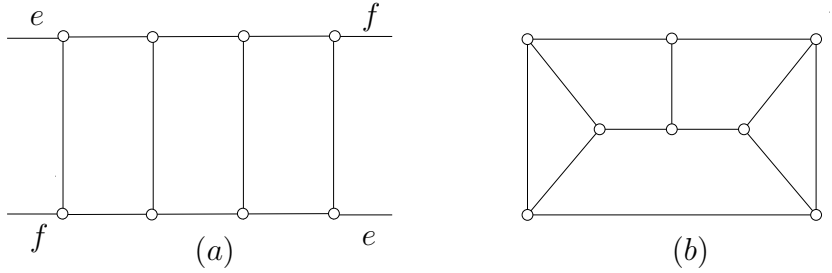,width=.8\textwidth}
      \caption{(a) a M{\"o}bius ladder, (b) the bicorn (a staircase)}
      \label{fig:ladders}
    \end{center}
  \end{figure}

  \subsubsection*{Staircases}
  Let $(u_1,u_2,\dots,u_k)$ and $(v_1,v_2,\dots, v_k)$ be two disjoint
  paths of length  at least two. The graph obtained  from the union of
  these  two paths  by adjoining  two new  vertices $x$  and $y$,  and
  joining  $u_i$ to  $v_i$,  for  $1\leq i\leq  k$,  and, in  addition
  joining $x$ to $u_1$ and $v_1$, $y$  to $u_k$ and $v_k$, and $x$ and
  $y$ to each other, is referred to as a {\em staircase} by Norine and
  Thomas~\cite{noth07}.  The  staircase on six vertices  is isomorphic
  to     the    triangular     prism.    The     graph    shown     in
  Figure~\ref{fig:ladders}(b)  is  the  staircase on  eight  vertices.
  Every staircase is a brick.

  \subsubsection{Near-bricks}
  Let $G$  be a matching covered  graph and let $C:=\partial(X)$  be a
  separating cut of  $G$ such that the subgraph $G[X]$ induced
  by $X$ is bipartite. As $C$  is a separating cut, by definition, the
  $C$-contraction   $G_1:=\Contra{G}{\overline{X}}{\overline{x}}$   is
  matching  covered and  thus  $|X|$ is  odd. So,  one  of the  colour
  classes of  $G[X]$ is larger  than the  other. We denote  the larger
  colour class  by $X_+$  and the  smaller colour  class by  $X_-$ and
  refer to  them, respectively, as  {\em majority} and  {\em minority}
  parts of $X$.  If the contraction vertex  $\overline{x}$ were joined
  by an  edge to  a vertex $v$  in the minority  part, then  the graph
  $G_1-\{\overline{x},v\}$  would be  a  bipartite  graph with  colour
  classes  of  different  cardinalities,  implying that  there  is  no
  perfect    matching   of    $G_1$    which    contains   the    edge
  $\{\overline{x},v\}$. This  is impossible because $G_1$  is matching
  covered.  The following  results  may be  easily  deduced from  this
  observation.

  \begin{prp}\label{prp:bipartite-shore}
    Let $C:=\partial(X)$  be a  separating cut  of a  matching covered
    graph $G$ such  that the subgraph $G[X]$ is bipartite.
Then  majority part $X_+$ of  $X$ is a barrier  of $G$,
    and $C$ is a tight cut associated with this barrier.\qed
  \end{prp}

  \begin{cor}\label{cor:cuts-in-bipartite-graphs}
    A cut of  a bipartite matching covered graph is separating if and only
    if it is tight.\qed
  \end{cor}

  \begin{cor}\label{cor:b=0}
    A matching covered  graph $G$ is bipartite if and only if \mbox{$b(G)=0$}. \qed
  \end{cor}

  \begin{cor}\label{cor:near-bricks}
    Let $G$ be a matching covered  graph with $b(G)=1$, and let $C$ be
    a tight  cut of $G$.  Then one of  the $C$-contractions of  $G$ is
    bipartite, the  majority part of that  shore is a barrier  of $G$,
    and $C$ is a barrier cut associated with that barrier.
  \end{cor}

  We refer  to a matching  covered graph $G$  with $b(G)=1$ as  a {\em
    near-brick}. Properties  of near-bricks are  in many ways  akin to
  those of bricks and, in trying to prove statements concerning bricks
  by  induction,  it   is  often  convenient  to  try   to  prove  the
  corresponding statements for near-bricks.

  \subsubsection{Bi-contractions and retracts}
  Suppose that $v_0$  is a vertex of degree two  in a matching covered
  graph $G$ of order four or more,  and let $v_1$ and $v_2$ denote the
  two    neighbours    of    $v_0$.    Then    $\partial(X)$,    where
  $X:=\{v_0,v_1,v_2\}$,   is   a  tight   cut   of   $G$.  The   graph
  $G/\overline{X}$  is  a brace  on  four  vertices,  and $G/X$  is  a
  matching covered  graph on  $|V(G)|-2$ vertices, and  is said  to be
  obtained by {\em bi-contracting} the vertex $v_0$ in $G$.

  \medskip

  Let $G$ be  any matching covered graph which has  order four or more
  and  is  not  an  even  cycle.   Then  one  can  obtain  a  sequence
  $(G_1,G_2,\dots,G_r)$ of  graphs such  that (i) $G_1=G$,  (ii) $G_r$
  has no  vertices of degree two,  and (iii) for $2\leq  i\leq r$, the
  graph $G_i$ is obtained from $G_{i-1}$ by bi-contracting some vertex
  of degree two  in it. Then, up to isomorphism,  the graph $G_r$ does
  not   depend   on   the  sequence   of   bi-contractions   performed
  (see~\cite{clm05}~Proposition 3.11).  We  denote it by $\widehat{G}$
  and refer to it as the {\em retract} of $G$. The retracts of the two
  graphs shown  in Figure~\ref{fig:bi-subdivision}  are, respectively,
  $K_4$  and  $\overline{C_6}$.   More  generally, the  retract  of  a
  bi-subdivision $H$ of a brick $J$ is $J$ itself.

  \smallskip

  Now we  define the operation of  bi-splitting a vertex which  may be
  regarded  as  the  converse  of   the  above  defined  operation  of
  bi-contraction. Given  any vertex  $v$ of  a matching  covered graph
  $H$,  we first  split  $v$ into  two new  vertices  $v_1$ and  $v_2$
  (called {\em outer vertices}), add  a third new vertex $v_0$ (called
  the {\em  inner vertex}), join  $v_0$ to  both $v_1$ and  $v_2$, and
  then distribute  the edges of  $H$ incident  to $v$ among  $v_1$ and
  $v_2$ in  such a  way that both  $v_1$ and $v_2$  have at  least two
  distinct   neighbours.    We   denote   the   resulting   graph   by
  $H\{v\rightarrow (v_1,v_0,v_2)\}$  and say that it  is obtained from
  $H$    by     {\em    bi-splitting}    the    vertex     $v$    (see
  Figure~\ref{fig:bisplitting}).    It    is   easy   to    see   that
  $H\{v\rightarrow (v_1,v_0,v_2)\}$  is a  matching covered  graph and
  that it has two more vertices and two more edges than $H$. Note that
  $H$  can  be  recovered from  $H\{v\rightarrow  (v_1,v_0,v_2)\}$  by
  bi-contracting the vertex $v_0$.

  \begin{figure}[!ht]
    \psfrag{H}{$H$}    
    \psfrag{G}{$H\{v\rightarrow   (v_1,v_0,v_2)\}$}
    \psfrag{v}{$v$}
    \psfrag{v0}{$v_0$}
    \psfrag{v1}{$v_1$}
    \psfrag{v2}{$v_2$}
    \begin{center}
      \epsfig{file=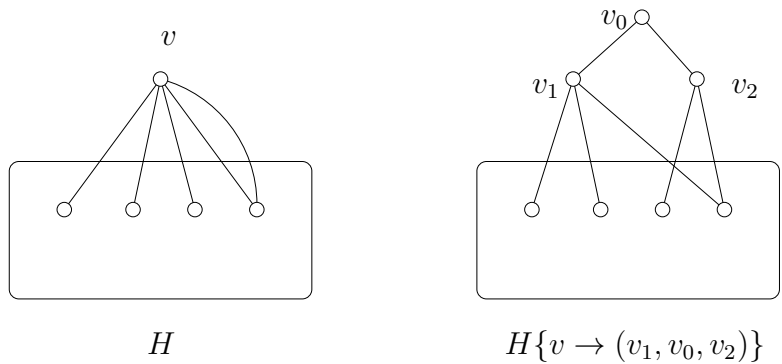, width=.75\textwidth}
      \caption{Bi-splitting a vertex in a matching covered graph}
      \label{fig:bisplitting}
    \end{center}
  \end{figure}

  \subsection{Solid bricks}
  A matching covered  graph is {\em solid} if every  separating cut of
  $G$  is a  tight  cut. 
  In particular, any  bipartite matching covered graph  is solid. That
  is, in a  bipartite matching covered graph, every  separating cut is
  also  a  tight  cut  (Corollary~\ref{cor:cuts-in-bipartite-graphs}).
  However, nonbipartite graphs, even  bricks, may have separating cuts
  which      are      not      tight.      (For      example,      see
  Figure~\ref{fig:sep-cuts}(b).)   {\em  Solid bricks}  are  precisely
  those bricks which are free of nontrivial separating cuts. It can be
  verified that  the graph shown in  Figure~\ref{fig:sep-cuts}(c) is a
  solid brick.
 The  following theorem  is  a consequence  of
  Corollary 2.26 in~\cite{clm02}.

  \begin{thm}\label{thm:solid-mc-graphs}
    A matching covered graph $G$ is solid if and only if each of its cut-contractions
with respect to any tight cut is also solid. (In particular, $G$ is solid if and only
each of its bricks is solid.)
  \end{thm}

  \medskip

  The notion  of solid matching  covered graphs was introduced 
  in~\cite{clm02} by  three of  us (CLM  -- Carvalho,  Lucchesi and
  Murty).  We  noted there  that certain  special properties  that are
  enjoyed by bipartite graphs are shared  by the more general class of
  solid  matching covered  graphs, and  exploited these  properties in
  establishing the validity of a conjecture due to Lov\'asz.

  \smallskip

  In  a  later  paper~\cite{clm04},  we (CLM)  showed  that  bipartite
  matching  covered graphs  and solid  near-bricks share  the property
  that their perfect  matching polytopes may be  defined without using
  the  odd  set  inequalities.   
  
%
%

  \smallskip
  The problem of recognizing solid bricks is in co-$\mathcal{NP}$,
  since any nontrivial separating cut serves as a certificate for demonstrating
  that a brick is nonsolid.
  In the same paper mentioned above, we (CLM) presented a proof of the
  following (unpublished) theorem due to two friends that provides
  another succinct certificate for demonstrating that a brick is nonsolid:

  \begin{thm}[B.~A.~Reed and Y.~Wakabayashi]
      \label{thm:usp}
    ~\\A brick $G$  has a nontrivial separating cut if  and only if it
      has  two disjoint  odd cycles  $C_1$ and  $C_2$ such  that $G  -
      (V(C_1) \cup V(C_2))$ has a perfect matching.~\qed
  \end{thm}

  
  We  showed in \cite{clm06} that the  only simple planar
  solid  bricks  are  the  odd  wheels.   The  solid-brick-recognition
  problem remains unsolved for nonplanar graphs.

  \begin{prb} \label{prb:nonsolid}
    Characterize solid bricks.~ (Is the problem of deciding whether
    or not a given brick is solid in the complexity class ${\cal NP}$?
    Is it in~${\cal P}$?)
  \end{prb} 

  As  stated in  the  abstract,  the objective  of  this  paper is  to
  establish  a connection  between this  unsolved problem  and another
  basic  problem (Problem  \ref{prb:prism-based}) concerning  matching
  covered graphs.

  \smallskip

  A graph is {\em odd-intercyclic} if any two odd cycles in it have a
  vertex in common. (Odd wheels and M\"obius ladder of order
  $4k$,  $k\geq  1$,  are  examples of  odd-intercyclic  bricks.)   It
  follows from Theorem~\ref{thm:usp}  that every odd-intercyclic brick
  is solid.

  \smallskip

  Kawarabayashi  and  Ozeki~\cite{kaoz13}  showed that  an  internally
  $4$-connected  graph  $G$  is  odd-intercyclic if  and  only  if  it
  satisfies one  of the following  conditions: (i) $G-v$  is bipartite
  for some  $v\in V$, (ii)  $G-\{e_1,e_2,e_3\}$ is bipartite  for some
  three  edges $e_1,e_2$  and $e_3$  which constitute  the edges  of a
  triangle of~$G$, (iii)  $|V|\leq 5$, or (iv)~$G$ can  be embedded in
  the projective plane so that each face boundary has even length.

  \smallskip

  The  above   result  leads   to  a  polynomial-time   algorithm  for
  recognizing odd-intercyclic  bricks.  However, not all  solid bricks
  are odd-intercyclic; the graph shown in Figure~\ref{fig:sep-cuts}(c)
  is a solid brick which is not odd-intercyclic. We have not been able
  to find a cubic solid brick that is not odd-intercyclic.

  \begin{cnj}\label{cnj:cubic-solid-bricks}
    Every cubic solid brick is odd-intercyclic.
  \end{cnj}

We conclude this section by defining an important parameter related to each
nontrivial separating cut $C$ of a nonsolid brick $G$.
Since $G$ is free of nontrivial tight cuts, it follows
that some perfect matching $M$ meets $C$ in at least three edges. We define
the {\em characteristic of $C$} to be the minimum value of $|M \cap C|$, where
the minimum is taken over all perfect matchings $M$ of $G$ that meet $C$ in
at least three edges.
(In particular, the characteristic of a nontrivial separating cut is at least three.)

  \subsection{Removable Classes}
  \subsubsection{Removable edges and doubletons}
  Let $G$ be a matching covered graph and let $e$ and $f$ be two edges
  of  $G$.  We   say  that  $e$  {\em  depends}  on   $f$,  and  write
  $e\Rightarrow f$, if every perfect matching of $G$ that contains $e$
  also contains $f$. Edges $e$ and $f$ are {\em mutually dependent} if
  $e\Rightarrow   f$    and   $f\Rightarrow    e$,   and    we   write
  $e\Leftrightarrow f$ to signify this. It
  is easy to see that  $\Leftrightarrow$ is an equivalence relation on
  the  edge set  $E(G)$  of $G$.  In general,  the  cardinality of  an
  equivalence class may be  arbitrarily large. (For example, $C_{2k}$,
  the cycle  of length $2k$, has  two equivalence classes of  size $k$
  each.) However, in  a brick, an equivalence classes has  at most two
  edges (see Theorem~\ref{thm:eq-classes-in-bricks}).

  \smallskip

  The  relation  $\Rightarrow$  may  be visualized  by  means  of  the
  directed graph  on the edge  set $E(G)$ of  $G$, where there  is an arc
  with $e$  as tail and  $f$ as  head whenever $e\Rightarrow  f$. From
  this  digraph  we  obtain  a  new digraph,  denoted  by  $D(G)$,  by
  identifying     equivalence    classes     under    the     relation
  $\Leftrightarrow$.  Clearly  $D(G)$  is  acyclic. We  refer  to  the
  equivalence classes that correspond to the sources of $D(G)$ as {\em
    minimal classes}.  For any edge  $e$ of $G$,  a source $Q$  of $D$
  that contains an edge  $f$ that depends on $e$ is said  to be a {\em
    minimal class induced by} $e$. (Here we admit the possibility that
  $e$ and $f$ may be the same.)

  \smallskip

  If  $R$ is  a minimal  class of  $G$, then  every edge  of $G-R$  is
  admissible. Moreover, if  $G-R$ happens  to be  connected then  $G-R$ is
  matching covered; in this  case, we  shall say that  $R$ is  a {\em
    removable class}.

  \smallskip

  An edge  $e$ of  a matching  covered graph $G$  is a  {\em removable
    edge} if $G-e$ is matching covered,  and a pair $\{e,f\}$ of edges
  of $G$  is a  {\em removable  doubleton} if neither  $e$ nor  $f$ is
  individually  removable,  but  the  graph  $G-\{e,f\}$  is  matching
  covered. In the former case, $\{e\}$  is a minimal class, and in the
  latter, $\{e,f\}$ is a minimal class.

The  result below concerning braces will  prove to  be useful.

  \begin{thm}(\cite{clm99},~Lemma 3.2)\label{thm:braces}
    Every edge in a brace of order six or more is removable.\qed
  \end{thm}

  \subsubsection{Removable classes in bricks}
  \label{sec:rem-bricks}
  A matching  covered graph $G$  is {\em  near-bipartite} if it  has a
  removable doubleton~$R$  such $G-R$ is a  bipartite matching covered
  graph.

  \begin{thm}[\protect{\cite[Lemma~2.3]{clm99}},
      \protect{\cite[Lemma~3.4]{lova87}}] 
    \label{thm:eq-classes-in-bricks}
    Any equivalence class $R$ in a brick $G$ has cardinality at most two.
Moreover,   if  $|R|=2$,  say
    $R=\{e,f\}$, then $G-e-f$  is a bipartite graph, both
    ends of $e$ are in one part of the bipartition of $G-e-f$ and both
    ends of $f$ are in the other part.
  \end{thm}

In particular, every removable class of a brick is either a removable edge
or is a removable doubleton. It follows from the above theorem that
every brick with a removable doubleton is indeed near-bipartite.
  Truncated  biwheels, prisms  of  order  $2\ ({\rm modulo\ }4)$,  M\"{o}bius
  ladders of  order $0\ ({\rm  modulo\ }4)$, and  staircases are  examples of
  near-bipartite  bricks.
  The  bicorn  (Figure~\ref{fig:four-bricks}(b))   has  two  removable
  doubletons,  and also a unique  removable edge.
The  two bricks  $K_4$ and  $\overline{C_6}$
  have three removable  doubletons each, but have  no removable edges;
the following was established by Lov{\'a}sz \cite{lova87}:

  \begin{thm}\label{removable-in-bricks}
    Every brick distinct from $K_4$ and $\overline{C_6}$ has a removable edge.
  \end{thm}

  There is an extensive discussion of removable edges in bricks in our
  paper \cite{clm12}.
We now present a technical  result which will
turn  out  to  be  useful  in  the proof  of  the  Main  Theorem (\ref{thm:four-nonsolid})
in Section~\ref{sec:conformal-minors}.

For a fixed vertex $v_0$ of a matching covered graph $G$,
a subset $M$ of the edges of $G$
is a {\em $v_0$-matching}
if $|M \cap \partial(v)|=1$ for each vertex $v$ distinct from $v_0$, and
if $|M \cap \partial(v_0)| > 1$.
A simple counting argument shows that $|M \cap \partial(v_0)|$ is odd.
The following may also be easily verified:

\begin{prp}
\label{prp:bipartite-v0-matching}
Let $G[A,B]$ be a bipartite graph such that $|A| = |B|$.
Then $G$ does not have a $v_0$-matching for any vertex $v_0$. \qed
\end{prp}

  \begin{lem}\label{lem:v0matching}
    Let $G$ be a brick, $v_0$ be a vertex of $G$, and $M$ be a $v_0$-matching.
Let  $e$   be  an   edge  in
    $\partial(v_0)-M$ and let $Q$ be a minimal class of $G$ induced by~$e$.
Then, $Q$ is a singleton which is disjoint from $M$.
  \end{lem}
  \begin{proof} If $e$ is the only member of $Q$ then there is nothing to prove.
Let $f$ be an edge of $Q$ such that
    $f\neq e$.  Then,  $f\Rightarrow e$ in $G$. As  $G-e$ has a perfect
    matching, and $f$ is inadmissible, by Proposition~\ref{prp:inadmissible},
     $G-e$ has  a barrier
    $B$ containing  both ends of  $f$, and $G-e-B$ has  exactly $|B|$
    odd components, two of which contain  the ends of $e$.
In particular, $v_0$  lies in  an odd component $K$  of
    $G-e-B$. 
    
    As  $|V(G)|$ is  even, $|M\cap\partial(v_0)|$ is  odd and
    so,    $|M\cap\partial(V(K))|$    is    also    odd.     Moreover,
    $|M\cap\partial(V(K'))|\ge  1$ for  any  other odd component $K'$  of
    $G-e-B$. By simple  counting, and taking into  account that $e\notin
    M$, we conclude that $|M\cap\partial(V(K'))|=1$, for each odd component
    $K'$ of $G-e-B$, and that each vertex of $B$ is matched by $M$ with a
    vertex in an odd component of $G-e-B$.  Thus, $f\notin M$. As $f$ is an
    arbitrary edge of $Q-\{e\}$, and since $e \notin M$,
we conclude that the minimal class $Q$ does not meet $M$.

It remains to argue that $|Q| = 1$.
By Theorem~\ref{thm:eq-classes-in-bricks}, $Q$ has at most two edges.     
Suppose that $|Q| = 2$.              By
    Theorem~\ref{thm:eq-classes-in-bricks},  $G-Q$  is  a  bipartite
    matching covered  graph.
Since $Q \cap M$ is empty,
$M$ is a $v_0$-matching of the bipartite graph $G-Q$,
and this contradicts Proposition~\ref{prp:bipartite-v0-matching}.
    Thus, $|Q|=1$.
  \end{proof}

%

  \subsubsection{Removable classes in solid graphs}
  Here we state some useful results regarding the properties and existence
of removable edges in solid graphs.

%
%
%


  \begin{thm}[\protect{\cite[Theorem~2.2.8]{clm02}}]
    \label{thm:solid-hereditary}
    For any removable edge $e$ of a solid matching covered graph~$G$,
the graph $G-e$ is also solid.
  \end{thm}

We now proceed to prove a result which we shall refer
to as the Lemma on Odd Wheels (\ref{lem:wheels}) which
 will play a crucial role in the proof of the main theorem
of Section~\ref{sec:conformal-minors}.
(A weaker version of this result appeared in \cite{clm02a}.)
The proof of this lemma relies on the following:

    \begin{thm}[\protect{\cite[Theorem~6.11]{clm12}}]
    \label{thm:solid-technical}
    Let $G$ be a  solid brick, let $v$ be a vertex of  $G$, let $n$ be
    the number of neighbours  of $v$, and let $d$ be  the degree of 
    $v$.  Enumerate the  $d$  edges of  $\partial(v)$ as  $e_i:=vv_i$,
    for 
    removable  $v$.  Enumerate  the  $d$  edges  of  $\partial(v)$  as
    $e_1,e_2,\ldots,e_d$,  where   $e_i$  joins  $v$  to   $v_i$,  for
    $i=1,2,\ldots,d$.   Assume   that  neither  $e_1$  nor   $e_2$  is
    removable in $G$.   Then, $n=3$ and, for $i=1,2$,  there exists an
    equipartition $(B_i,I_i)$ of $V(G)$ such that
    \begin{enumerate}
      \rmenum
    \item
      $e_i$ is the only edge of $G$ that has both ends in $I_i$,
    \item
      \label{item:incident-v3}
      every edge that  has both ends in $B_i$ is  incident with $v_3$,
      and
    \item
      the subgraph $H_i$ of $G$, obtained  by the removal of $e_i$ and
      each edge  having both  ends in $B_i$,  is matching  covered and
      bipartite, with bipartition $\{B_i,I_i\}$.
    \end{enumerate}
    Moreover,  $B_1=(I_2-v)   \cup  \{v_3\}$  and   $B_2=(I_1-v)  \cup
    \{v_3\}$. (See Figure~\ref{fig:B-I} for an illustration.)
  \end{thm}
    \begin{figure}[htb]
      \begin{center}
	\psfrag{B2}{$B_2$} 
	\psfrag{B1}{$B_1$} 
	\psfrag{I1}{$I_1$}
	\psfrag{I2}{$I_2$} 
	\psfrag{e1}{$e_1$} 
	\psfrag{e2}{$e_2$}
	\psfrag{v}{$v$}
	\psfrag{v1}{$v_1$}
	\psfrag{v2}{$v_2$}
	\psfrag{v3}{$v_3$}
	\psfrag{u}{$u$}
	\psfrag{w}{$w$}
	\psfrag{x}{$u$}
	\psfrag{y}{$w$}
	\epsfig{file=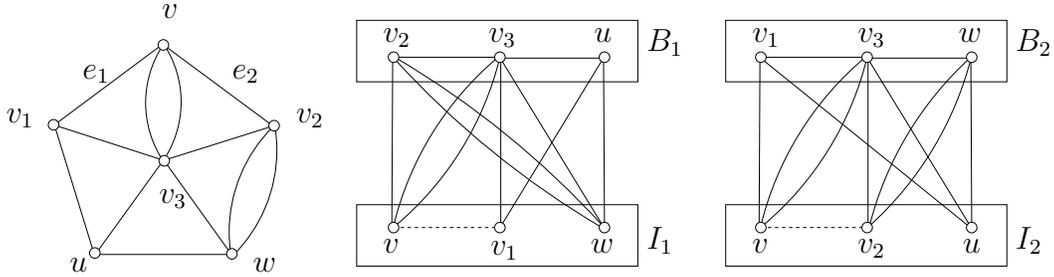, width=\textwidth}
	\caption{Graphs $G$, $G-e_1$ and $G-e_2$} 
	\label{fig:B-I}
      \end{center}
    \end{figure}

  We  say  that  $G$  is  a
  \emph{$v_0$-wheel} if  $G$ is a wheel  having $v_0$ as a  hub.

  \begin{lem}[Lemma on Odd Wheels]
    \label{lem:wheels}
    Let $G$ be a simple solid brick, $v_0$ be a vertex of $G$,
and $M_0$ be a  $v_0$-matching.
Then either $G$  is a $v_0$-wheel or $G$ has a
    removable edge $e \notin M_0 \cup \partial(v_0)$.
  \end{lem}
  \begin{proof}
    \setcounter{cas}{0}
    \begin{cas}
      The brick $G$ has  a vertex $v \ne v_0$ that  has degree four or
      more in $G$.
    \end{cas}
    As $G$ is simple, at least two  edges, $e_1$ and $e_2$, are not in
    $M_0   \cup  \partial(v_0)$   but  are   incident  with   $v$.  By
    Theorem~\ref{thm:solid-technical},  one  of  $e_1$  and  $e_2$  is
    removable in $G$.

   \bigskip
    We may thus assume that every vertex $v \ne v_0$ has degree three
    in $G$.

    \begin{cas}
      Every vertex of $G-v_0$ is adjacent to $v_0$.
    \end{cas}
    Since every vertex $v \ne v_0$ has degree three in $G$ and is adjacent to $v_0$,
    every vertex distinct from $v_0$ has degree two in $G-v_0$.
    Then $G-v_0$ is a collection of cycles. By  the 3-connectivity of $G$,  it follows
    that $G-v_0$ is a cycle and, consequently, $G$  is a $v_0$-wheel.
    

    \begin{cas}
      The previous cases are not applicable.
    \end{cas}
    Every vertex $v \ne v_0$ of $G$ has degree three in $G$. Moreover,
    $G$ has a vertex, $v \ne v_0$,  that is not adjacent to $v_0$. Let
    $e_i:=vv_i$,  $i=1,2,3$,   be  the   three  edges   incident  with
    $v$. Adjust notation so that $e_3 \in M_0$.

    \begin{sta}
      One of the edges $e_1$ and $e_2$ is removable in $G$.
    \end{sta}
    \begin{proof}
      Assume the  contrary. By  Theorem~\ref{thm:solid-technical}, $G$
      has an  equipartition $(B_1,I_1)$  such that  $e_1$ is  the only
      edge having both  ends in $I_1$ and every edge  having both ends
      in $B_1$ is  incident with $v_3$. Moreover,  the bipartite graph
      $H$ obtained  from $G$  by the  removal of  $e_1$ and  each edge
      having both ends  in $B_1$ is matching covered.  Vertex $v_3$, a
      vertex adjacent to $v$, is  distinct from $v_0$. Thus, $v_3$ has
      degree three. As $H$ is  matching covered, precisely one edge of
      $\partial(v_3)$, say $f$,  has both ends in $B_1$.  But $e_3$ is
      the only edge of  $M_0$ incident with $v$ and its  end $v$ is in
      $I_1$. Thus, $f \not \in  M_0$.
      In particular, $M_0$ is a $v_0$-matching of $H$,
      and this contradicts Proposition~\ref{prp:bipartite-v0-matching}.
    \end{proof}

    The proof of the Lemma on Odd Wheels is complete.
  \end{proof}

  \subsection{Ear Decompositions}
  \subsubsection{Deletions and additions of ears}
  A path $P:=v_0v_1\dots  v_{\ell}$ of odd length in a  graph $G$ is a
  {\em single ear} \underline{in} $G$ if each of its internal vertices
  $v_1,v_2,\dots,v_{\ell-1}$  has degree  two  in $G$.   If $P_1$  and
  $P_2$ are two vertex-disjoint single ears in $G$, then $\{P_1,P_2\}$
  is a {\em  double ear} with $P_1$ and $P_2$  as its {\em constituent
    single ears}.   The {\em deletion}  of a  single ear $P$  from $G$
  consists of deleting all the internal vertices of
  $P$, and the  graph obtained by deleting $P$ from  $G$ is denoted by
  $G-P$. Likewise, the deletion of a double ear $\{P_1,P_2\}$ consists
  of deleting each of its constituent single ears $P_1$ and $P_2$.

  \smallskip

  A single ear $P$ in a  matching covered graph $G$ is {\em removable}
  if  the graph  $G-P$  obtained  by deleting  $P$  from  $G$ is  also
  matching covered.  If $P_1$ and  $P_2$ are two vertex-disjoint  single ears
  neither of which is removable,  but the graph \mbox{$G-P_1-P_2$} is matching covered,
  then  the double  ear  $\{P_1,P_2\}$ is  {\em  removable}. When  the
  length of  a single ear  is one, then we  identify it with  its only
  edge.

  \smallskip

  The  following basic  result  concerning  ear  decompositions  was  proved  by
  Lov\'asz and Plummer~\cite{lopl86}:

  \begin{thm}[{\sc The two-ear Theorem}]\label{thm:ear-decomp}
    ~ Given  any matching covered  graph $G$, there exists  a sequence
    $(G_1,G_2,\dots,G_r)$ of  matching covered  subgraphs of  $G$ such
    that:
    \begin{enumerate}
    \item
      $G_1=K_2$ and $G_r=G$; and
    \item
      for $2\leq i\leq r$, the  graph $G_{i-1}$ is obtained from $G_i$
      by  the deletion  of  either  a removable  single  ear  or of  a
      removable double ear.
    \end{enumerate}
  \end{thm}


  \subsubsection{Conformal subgraphs}
  A matching covered subgraph $H$ of a matching covered graph $G$ is
  {\em   conformal}   if   the    graph   $G-V(H)$   has   a   perfect
  matching. It is easily seen that this notion obeys transitivity:

  \begin{prp}\label{prp:transitivity}
Any conformal subgraph of a conformal subgraph of a matching covered graph $G$
is also a conformal subgraph of $G$. \qed
  \end{prp}

Conformal subgraphs have been referred to by various other
  names  (`nice'  sugraphs,   `central'  subgraphs  and  `well-fitted'
  subgraphs)  in  the  literature.  The following  result  is  due  to
  Lov\'asz and Plummer~\cite{lopl86}.

  \begin{thm}\label{thm:conformal}
    A matching covered subgraph  $H$ of  a matching  covered graph  $G$ is  conformal
    if  and only  if there is  some ear  decomposition ${\cal
      G}:=(G_1,G_2,\dots,G_r)$  of $G$  such that  $H$ is  one of  the
    graphs in ${\cal G}$.\qed
  \end{thm}

  It  follows from  the  above  theorem that  if  $H$  is a  conformal
  matching covered subgraph of a  matching covered graph $G$,
then $H$  can be obtained
  from $G$  by a sequence  of deletions  of removable ears  (single or
  double). But  the deletion of an  ear amounts to first  reducing that
  ear  to one of length  one by  means of  bi-contractions, and  then
  deleting the  only edge  of that ear.  This observation  implies the
  following:

  \begin{cor}\label{cor:conformal}
    A matching covered subgraph  $H$ of  a matching  covered graph  $G$ is conformal
    if  and  only  if  it   can  be  obtained  from  $G$  by
    bi-con\-trac\-tions  of  vertices  of  degree  two and deletions  of
    removable classes.\qed
  \end{cor}

  Every  ear  decomposition  of  a bipartite  matching  covered  graph
  involves only single ear additions. However, in an ear decomposition
  of a nonbipartite matching covered graph  there must be at least one
  double  ear addition.  Lov\'asz established  the following  fundamental
  result concerning nonbipartite graphs:

  \begin{thm}[\cite{lova83}]
    \label{thm:canonical-ear-decomp}
    Every nonbipartite matching covered graph has an ear decomposition
    such that either  the third graph is a bi-subdivision  of $K_4$ or
    the fourth graph is a bi-subdivision of $\overline{C_6}$.
  \end{thm}

  This theorem gives rise to  natural questions which are described in
  terms  of special  types of minors of matching  covered graphs  which we  now
  proceed to discuss.

  \subsubsection{Conformal minors and matching minors}
  Let $G$ be a matching covered graph. A matching covered graph $J$ is
  a {\em conformal minor} of $G$  if some bi-subdivision $H$ of $J$ is
 a conformal  subgraph of  $G$. Figure~\ref{fig:minors-of-P-and-P+e}(a)
  shows that $K_4$ is a  conformal minor of $\mathbb{P}$, the Petersen
  graph.  It is not too difficult to show that $\overline{C_6}$ is not
  a conformal minor of $\mathbb{P}$. However, if $\mathbb{P}+e$ is any
  graph obtained  by adding  an edge $e$  to $\mathbb{P}$  joining two
  nonadjacent vertices, then $\overline{C_6}$  is a conformal minor of
  $\mathbb{P}+e$             as             illustrated             in
  Figure~\ref{fig:minors-of-P-and-P+e}(b),   and  also   of  $K_4\odot
  \mathbb{P}$              as              illustrated              in
  Figure~\ref{fig:minors-of-P-and-P+e}(c).

  \begin{figure}[!ht]
    \psfrag{a}{$(a)$}        \psfrag{b}{$(b)$}       \psfrag{c}{$(c)$}
    \psfrag{e}{$e$}
    \begin{center}
      \epsfig{file=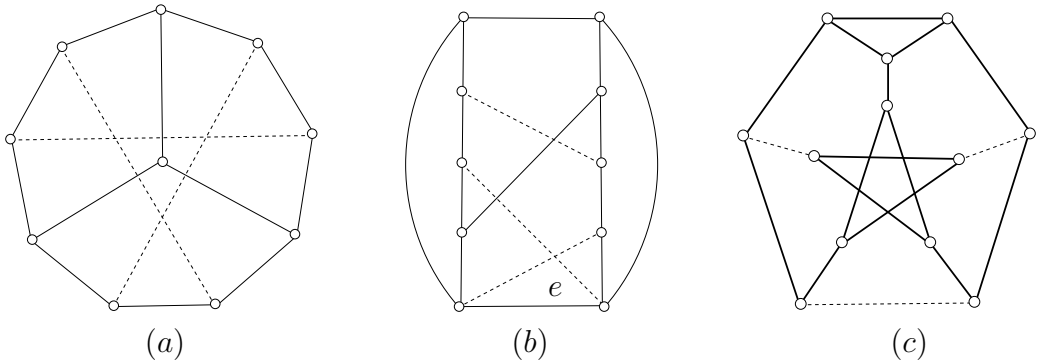, width=\textwidth}
      \caption{Conformal  minors:  (a)   $K_4$  of  $\mathbb{P}$;  (b)
	$\overline{C_6}$  of $\mathbb{P}+e$;  (c) $\overline{C_6}$  of
	$K_4\odot\mathbb{P}$}
      \label{fig:minors-of-P-and-P+e}
    \end{center}
  \end{figure}

  Given a  fixed matching covered graph  $J$, we say  that a
  matching covered graph $G$ is  {\em $J$-based} if $J$ is a
  conformal minor of  $G$, and, otherwise $G$ is  \mbox{\em $J$-free}.  For
  example,  the  Petersen graph  $\mathbb{P}$  is  $K_4$-based but  is
  $\overline{C_6}$-free,     and     $\mathbb{P}+e$    depicted     in
  Figure~\ref{fig:minors-of-P-and-P+e}(b)  is   both  $K_4$-based  and
  $\overline{C_6}$-based.
  
  \smallskip
  Theorem~\ref{thm:canonical-ear-decomp}
  implies  that every  nonbipartite matching  covered graph  is either
  $K_4$-based or is $\overline{C_6}$-based (or both),
  and it raises two natural problems: characterize those matching
  covered graph that are $K_4$-free, and those that are $\overline{C_6}$-free.
%
  Two of us (KM -- Kothari and Murty) showed  that it suffices to  solve these
  problems for  bricks  by establishing  the following  result
  concerning cubic bricks.

  \begin{thm}[\cite{komu16}]\label{thm:km-cubic}
    Suppose  that $J$  is a  cubic brick  and that  $C$ is a tight cut of a matching
    covered graph  $G$. Then $G$ is $J$-free if and only if each $C$-contraction
    of $G$ is $J$-free. (In particular, $G$ is $J$-free if and only if each brick
    of $G$ is $J$-free.)
  \end{thm}

  The  restriction that  $J$  be  a cubic  brick  is  crucial for  the
  validity of  the above statement.  (Curiously, it is not  valid even
  for  cubic braces.  For example,  consider the  graph $G:=K_{4}\odot
  K_{3,3}$. If $C$ denotes the unique nontrivial tight cut in $G$, one
  of  the $C$-contractions  of $G$  is the  brace $K_{3,3}$.  However,
  $K_{3,3}$ is not a conformal minor of $G$!)
  

\smallskip
In light of Theorem~\ref{thm:km-cubic}, it suffices to solve the
following problems:

\begin{prb}\label{prb:K4-based}
    Characterize $K_4$-free bricks.
  \end{prb}

  \begin{prb}\label{prb:prism-based}
    Characterize $\overline{C_6}$-free bricks.
  \end{prb}

  Using the brick generation theorem of Norine and  Thomas,
  which  will be described  later on,  we (KM)        were        able        to        resolve
  Problems \ref{prb:K4-based}~and~\ref{prb:prism-based} in the special
  case  of planar  bricks  by  proving the  following  results. (By  a
  well-known  theorem of  Whitney (1933),  every simple  $3$-connected
  planar graph has  a unique embedding in the  plane (Theorem~10.28 in
  \cite{bomu08}.)

  \begin{thm}[{\cite{komu16}}]\label{thm:K4-free-planar}
    A  simple planar  brick is  $K_4$-free if  and only  if
its (unique) planar embedding has precisely two odd faces.
  \end{thm}

  \begin{thm}[\cite{komu16}]
    \label{thm:prism-free-planar}
    The only  simple planar $\overline{C_6}$-free bricks are  the odd
    wheels, staircases of order $4k$, and the tricorn.
  \end{thm}

In the case of nonplanar bricks, Problems~\ref{prb:K4-based} and \ref{prb:prism-based} remain unsolved.

\smallskip
  Norine and Thomas~\cite{noth07} call a  matching covered graph $J$ a
  {\em matching minor}  of a matching covered graph $G$  if $J$ can be
  obtained  from  a  conformal  subgraph   $H$  of  $G$  by  means  of
  bi-contractions.   By  Corollary~\ref{cor:conformal}, any  conformal
  subgraph  $H$  of  $G$  can  be   obtained  from  $G$  by  means  of
  bi-contractions and deletions of removable classes. And, if $H$ were
  a bi-subdivision of  $J$, then $J$ can clearly be  obtained from $H$
  by means of bi-contractions. (A {\em restricted bi-contraction} is a
  bi-contraction of  a vertex  of degree two ---  one of  whose neighbours
  also has degree two. If $H$ is a bi-subdivision of $J$, then $J$ can
  in fact be obtained from $H$ by restricted bi-contractions.)

  \smallskip

  It follows from the above observations that every conformal minor of
  $G$ is  also a  matching minor of  $G$. But the  converse is  not true in
  general (due to  the  fact  unrestricted bi-contractions  are
  permissible in obtaining a matching  minor of $G$.) For example, the
  wheel  $W_5$  is  a  matching  minor  of  the  graph  $G$  shown  in
  Figure~\ref{fig:matching-minor} because  $W_5$ can be  obtained from
  $G$  by first  deleting the  edge  $e$ and  then bi-contracting  the
  vertex $v$ in the resulting graph.   But it is not a conformal minor
  of $G$ for the simple reason that $G$ has no vertices of degree greater than four.

However, it is easily seen that if a cubic matching
covered graph $J$ is obtained from a matching covered graph $H$
by means of bi-contractions, then $H$ must be a bi-subdivision of $J$.
Consequently, we have the following: 

  \begin{cor}\label{cor:two-types-of-minors}
    A cubic matching covered graph $J$ is a  matching minor of a matching covered graph
    $G$ if and only if $J$ is a conformal minor of $G$.
  \end{cor}

  \begin{figure}[!ht]
    \psfrag{W5}{$W_5$} 
    \psfrag{G}{$G$} 
    \psfrag{e}{$e$} 
    \psfrag{v}{$v$}
    \begin{center}
      \epsfig{file=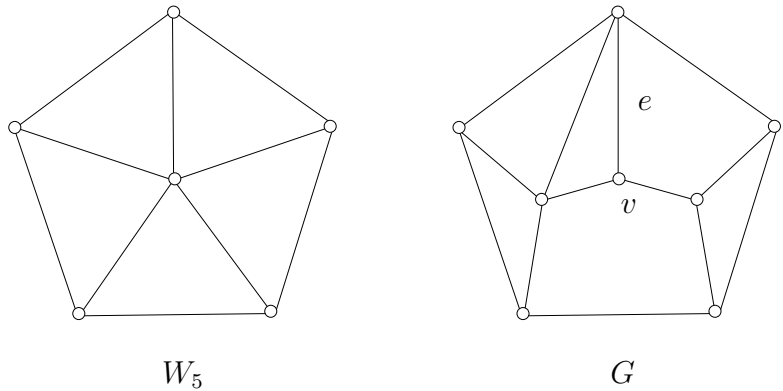, width=.75\textwidth}
      \caption{The wheel $W_5$ is a matching minor of $G$}
      \label{fig:matching-minor}
    \end{center}
  \end{figure}

  We      conclude      this      section     by      noting      that
  Theorems~\ref{thm:solid-mc-graphs}    and~\ref{thm:solid-hereditary}
  together  imply the following:

\begin{cor}
\label{cor:conformal-minor-solidity}
Every conformal minor of a solid matching covered graph is a solid matching covered graph.
\end{cor}


However, not every conformal  minor of a nonsolid graph
  is nonsolid. For  example, the bicorn is nonsolid,  but $K_4$, which
  is solid,  is a conformal minor  of the bicorn.

Since $\overline{C_6}$ is nonsolid, we have the following consequence:

\begin{cor}
\label{cor:solid-subset-of-C6bar-free}
Every solid matching covered graph is $\overline{C_6}$-free.
\end{cor}


  \subsection{Robust cuts in bricks}
  \subsubsection{$b$-invariant edges}
  Recall  that, for a matching  covered graph  $G$,  the symbol  $b(G)$
  represents the  number of bricks  in any tight cut  decomposition of
  $G$.

  \smallskip

  A  removable edge  $e$  of a  brick $G$  is  {\em $b$-invariant}  if
  $b(G-e)=b(G)=1$.  Motivated by  his  work on  the matching  lattice,
  Lov\'asz~\cite{lova87}  had   conjectured  that  every  brick
    distinct  from  $K_4$,  $\overline{C_6}$ and  $\mathbb{P}$  has  a
    $b$-invariant   edge.    All   bricks  other   than   $K_4$   and
  $\overline{C_6}$ have  removable edges; in  fact, every edge  of the
  Petersen  graph is  removable.   What is  striking about  Lov\'asz's
  conjecture is that it asserts that among bricks which have removable
  edges, Petersen graph  is the only brick which  has no $b$-invariant
  edges.   In  \cite{clm02a},  three  of   us  presented  a  proof  of
  a strengthening of Lov\'asz's cojecture. We showed:

  \begin{thm}\label{thm:b-invariant-brick}(\cite{clm02a})
    Every brick distinct from $K_4$, $\overline{C_6}$, the bicorn, and
    the Petersen graph $\mathbb{P}$ has two $b$-invariant edges.
  \end{thm}

  Many  of the  notions and  results  used in  the proofs  of the  new
  results in  this paper  arose in  that context of  our proof  of the
  above theorem.  We record below  only those definitions  and results
  that are essential for proving the main result of this paper.

  \subsubsection{Existence of robust cuts}
  The starting point of  our attempt to resolve  Lov\'asz's conjecture was
  to investigate the implications of the existence of a removable edge
  in  a brick  which fails  to be  $b$-invariant. This  led us  to the
  serendipitous discovery of solid bricks. We were able to show:

  \begin{thm}\label{thm:solid-b-invariant}(\cite{clm02})
    Any brick which has a removable edge that is not \mbox{$b$-invariant} has
    a nontrivial  separating cut. (Consequently, every  removable edge
    in a solid brick is $b$-invariant.)
  \end{thm}

  It  follows  from  the  above  theorem that  it  suffices  to  prove
  Lov\'asz's conjecture  for nonsolid bricks. As  every nonsolid brick
  $G$ has  nontrivial separating cuts, it  was natural to try  to show
  that  $G$  has  a  $b$-invariant  edge  inductively  by  considering
  cut-contractions  of  $G$  with  respect to  a  suitable  separating
  cut. This idea led us to the notion of a robust cut.

  \smallskip

  A nontrivial separating cut $C:=\partial(X)$ of a nonsolid brick $G$
  is   {\em   robust}  if   both   the   $C$-contractions  $G/X$   and
  $G/\overline{X}$ of $G$  are near-bricks.
We say that a robust cut $C$ is {\em $k$-robust} if $C$ has characteristic $k$.
We were able  to prove the
  following fundamental result.

  \begin{thm}[\protect{\cite[Theorem~4.1]{clm02a}}]
    \label{thm:3-robust}
    Every simple brick distinct from the Petersen graph has a 3-robust
    cut. The Petersen graph  has only 5-robust cuts.\footnote{In fact,
      in every  simple brick distinct  from the Petersen  graph, every
      robust cut is 3-robust~\cite{calu00}.}
  \end{thm}

  Both   the   cut   $C$   and   $D$  in   the   brick   depicted   in
  Figure~\ref{fig:robust-non-robust} are separating cuts, but only one
  of them, namely $C$, is a robust cut.

  \begin{figure}[!ht]
    \psfrag{C}{$C$}\psfrag{D}{$D$}
    \begin{center}
      \epsfig{file=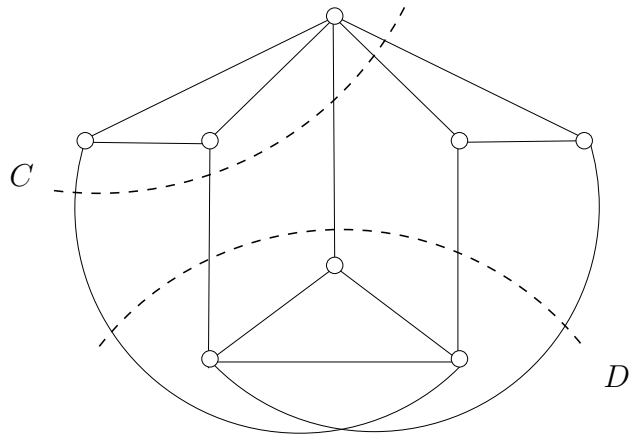, width=.6\textwidth}
      \caption{Cut $C$ is robust, but $D$ is not}
      \label{fig:robust-non-robust}
    \end{center}
  \end{figure}

  \section{Conformal Minors of Nonsolid Bricks }
  \label{sec:conformal-minors}

%

  We shall refer to $\overline{C_6}$, the bicorn, the tricorn and the Petersen graph
  as the {\em basic} nonsolid bricks.

  \begin{thm}[{\sc Main Theorem}]
    \label{thm:four-nonsolid}
    Every nonsolid  matching covered  graph contains a  basic nonsolid
    brick as a conformal minor.
  \end{thm}
  \begin{proof}
    Let $G$  be any nonsolid  matching covered graph.  We  shall prove
    the  validity of  the assertion by  induction on  the
    number of edges of~$G$. 
    
    It follows from Theorem~\ref{thm:canonical-ear-decomp}
    that the smallest nonsolid matching covered graph is $\overline{C_6}$,
    which is a basic nonsolid brick. 
    For the general case, we adopt as the inductive
    hypothesis that  every nonsolid matching covered  graph with fewer
    edges than  $G$ has  one of  the four basic  nonsolid bricks  as a
    conformal minor.

    \setcounter{cas}{0}
    \begin{cas}
      $G$ contains a proper conformal  subgraph $H$ that is a nonsolid
      matching covered graph
    \end{cas}
    By the induction  hypothesis, $H$ contains a  basic nonsolid brick
    as        a        conformal       minor.         Hence,        by
    Proposition~\ref{prp:transitivity},  $G$  also  contains  a  basic
    nonsolid brick as a conformal minor.

    \smallskip

    Note that this case applies when $G$ has multiple edges.

    \begin{cas}
      Graph $G$ has a nontrivial tight cut $C$.
    \end{cas}
    By Theorem~\ref{thm:solid-mc-graphs},  $G$ has  a $C$-contraction,
    $G_1$,  that  is nonsolid.   By  the  induction hypothesis,  $G_1$
    contains a basic  nonsolid brick as a conformal  minor. Since each
    of  the   basic  nonsolid  bricks   is  cubic,  it   follows  from
    Theorem~\ref{thm:km-cubic} that $G$ also contains a basic nonsolid
    brick as a conformal minor.

    \begin{cas}
      Previous cases do not apply.
    \end{cas}
    The  graph $G$  is free  of nontrivial  tight cuts,  hence $G$  is
    either a brick or a brace.  Every bipartite graph is solid.  Thus,
    $G$ is a brick.  In fact, $G$ is a simple  nonsolid brick, free of
    nonsolid conformal minors. In sum,

    \begin{lem}
      \label{lem:G-R-solid}
      Let $R$ be a nonempty set of  edges of $G$. If $G-R$ is matching
      covered then it is solid.  \qed
    \end{lem}

    We  shall  prove that  $G$  is  one  of  the four  basic  nonsolid
    bricks. If $G$ is the Petersen graph then we are done. We may thus
    assume that $G$ is not the Petersen graph. We shall prove that $G$
    is either $\overline{C_6}$, the bicorn or the tricorn.
    
    As    $G$    is    not    the    Petersen    graph,    then,    by
    Theorem~\ref{thm:3-robust},   $G$   has    3-robust   cuts.    Let
    $C:=\partial(X)$  be a  3-robust cut  of $G$  and let  $M_0$ be  a
    perfect  matching  of  $G$  such   that  $|M_0  \cap  C|=3$.   Let
    $G_1:=\Contra{G}{\overline{X}}{\overline{x}}$                  and
    $G_2:=\Contra{G}{X}{x}$  be   the  two  $C$-contractions   of  $G$
    obtained by contracting $\overline{X}$  and $X$ to single vertices
    $\overline{x}$ and $x$, respectively. As $C$ is robust, the graphs
    $G_1$ and $G_2$ are near-bricks.

    \begin{lem}
      \label{lem:R-rem-both=>tight-solid}
      Let $R$ be a nonempty set of edges  of $G$.  If $G_1-R$  and $G_2-R$ are
      both matching  covered then the  graphs $G_1-R$ and  $G_2-R$ are
      both solid and $C - R$ is tight in $G-R$.
    \end{lem}
    \begin{proof}
      Suppose   that   $G_1-R$   and   $G_2-R$   are   both   matching
      covered. Then, $G-R$  is matching covered, and the  cut $C-R$ is
      separating  in $G-R$.   By Lemma~\ref{lem:G-R-solid},  the graph
      $G-R$ is solid.   Thus, $C - R$ must be  tight in $G-R$.
      Moreover,   by   Theorem~\ref{thm:solid-mc-graphs},  
both   $(C- R)$-contractions of  $G-R$ must  be solid.   That is,
      $G_1-R$ and $G_2-R$ are both solid.
    \end{proof}

    \begin{cor}
      \label{cor:rem=>M0}
      Let $e$  be an  edge of  $G$.  If $G_1-e$  and $G_2-e$  are both
      matching covered  then $e \in  M_0$ and $G_1-e$ and  $G_2-e$ are
      both solid.
    \end{cor}
    \begin{proof}
      Suppose that  the graphs $G_1-e$  and $G_2-e$ are  both matching
      covered.  By  Lemma~\ref{lem:R-rem-both=>tight-solid},  the  cut
      $C-e$ is tight  in $G-e$. Thus, $M_0$ is not  a perfect matching
      of   $G-e$,   hence   $e    \in   M_0$.    Moreover,   also   by
      Lemma~\ref{lem:R-rem-both=>tight-solid}, the  graphs $G_1-e$ and
      $G_2-e$ are both solid.
    \end{proof}

    \begin{lem}
      The graphs $G_1$ and $G_2$ are bricks.
    \end{lem}
    \begin{proof}
      Suppose that $G_1$ is not a brick. As $G_1$ is a near-brick that
      is  not a  brick, it  has nontrivial  tight cuts.   Moreover, by
      Corollary~\ref{cor:near-bricks}, for any tight cut $D$ of $G_1$,
      one  of the  $D$-contractions of  $G_1$ must  be bipartite,  the
      other must be a near-brick.
      \begin{sta}
	\label{sta:one-in-each}
	Let $D$ be a nontrivial tight cut of $G_1$, and let $Y$ be its
	nonbipartite shore. The vertices  $\overline{x}$ and ${y}$ lie
	in     distinct     parts     of    the     bipartition     of
	$H:=\Contra{G_1}{{Y}}{{y}}$.
      \end{sta}
      \begin{proof}
	The graph $H$ is bipartite  and matching covered.  Any part of
	$H$  that   is  disjoint  with  $\{\overline{x},{y}\}$   is  a
	nontrivial barrier of $G$. Thus,  not only is $\overline{x}$ a
	vertex of $H$, but also it lies in the part of $H$ that does not contain
      vertex $y$.
      \end{proof}

      Consider now a  tight cut decomposition of $G_1$.  For any tight
      cut $D$  of $G_1$,  every tight cut  in some  $D$-contraction of
      $G_1$ is also a tight cut  of $G_1$.  Using this observation and
      applying~\eqref{sta:one-in-each}  repeatedly,  we conclude  that
      there exists a nested sequence
      \begin{displaymath}
	X_0 \subset X_1 \subset \cdots \subset X_r = X \quad (r \ge 1)
      \end{displaymath}
      of subsets of $X$ such that the $r$ cuts $\partial(X_i)$, $0 \le
      i < r$ are the tight cuts used in the tight cut decomposition of
      $G_1$,   and,    for   $1    \le   i    \le   r$,    the   graph
      $H_i:=\Contra{(\Contra{G}{X_{i-1}}{x_{i-1}})}
      {\overline{X_i}}{\overline{x_i}}$ is  a brace  of order  four or
      more. Moreover, $G_0:=G/\overline{X_0}$ is a brick.

      \begin{figure}[ht]
  	\psfrag{X}{\hspace{-3pt}$\overline{X_0}$}
	\psfrag{Yb}{\hspace{-3pt}$X_1$}
	\psfrag{v}{$v$}
  	\begin{center}
  	  \epsfig{file=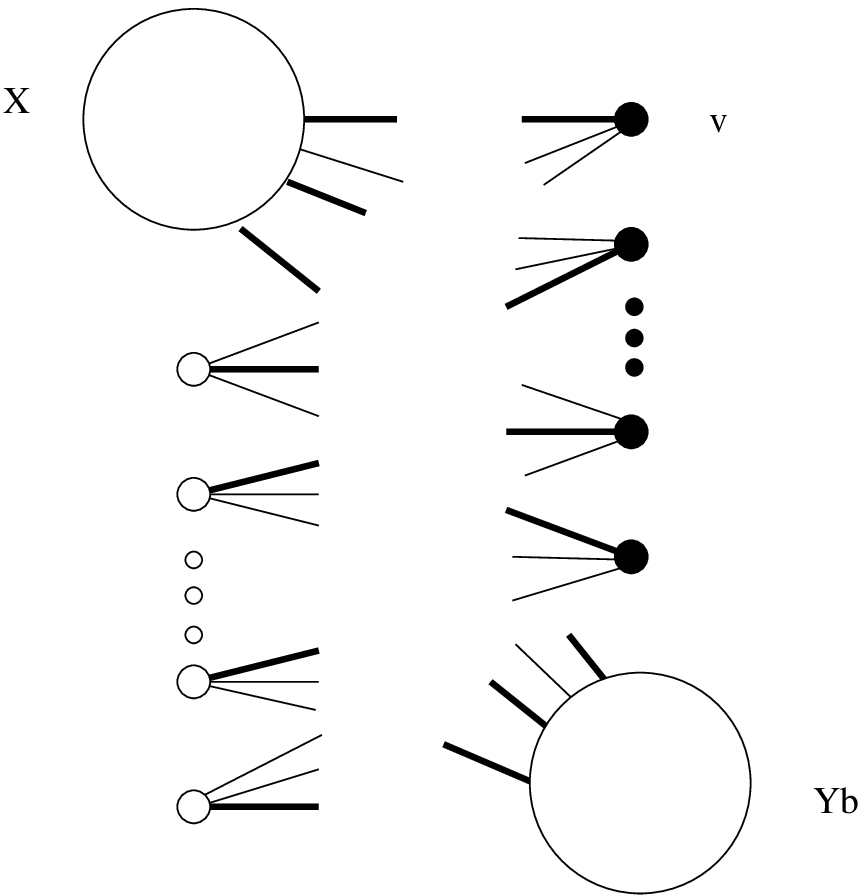, width=.4\textwidth}
  	  \caption{The     brace      $H_1:=\Contra
	    {\Contra{G}{X_0}{x_0}} {\overline{X_1}}{\overline{x_1}}$}
  	  \label{fig:H1}
  	\end{center}
      \end{figure}

      Let  us  analyze   the  situation  in  the   brace  $H_1$.   See
      Figure~\ref{fig:H1}.   By~\eqref{sta:one-in-each}, the  vertices
      $x_{0}$  and   $\overline{x_1}$  lie  in  distinct   parts  of
      $H_1$. Let  $v$ denote a vertex  in the same part  that contains
      $\overline{x_1}$, but distinct from $\overline{x_1}$.  
      Then, no edge incident with $v$ is in $C$.

      Consider first the case in which  $v$ is adjacent to at most one
      vertex of~$X_0$.   As $G$ is  a brick,  then $v$ is  adjacent to
      three or more vertices of  the brace~$H_1$.  Thus, $H_1$ has six
      or more vertices. By  Theorem~\ref{thm:braces}, all the edges of
      $H_1$ are  removable. In particular,  one of the edges  of $H_1$
      incident with  $v$ is not  in $M_0 \cup  \partial(X_{0})$. Thus,
      $G_1$ has  a removable edge that  does not lie in  $M_0 \cup C$.
      This is a contradiction to Corollary~\ref{cor:rem=>M0}.

      Alternatively, suppose that  $v$ is adjacent to  two vertices of
      $X_0$, say,  $w_1$ and $w_2$.   The edges $vw_1$ and  $vw_2$ are
      multiple  edges in  $G/X_0$, hence  removable in  $G_1/X_0$.  At
      least  one of  the  edges $vw_1$  and $vw_2$  is  not in  $M_0$.
      Adjust    notation   so    that   $vw_1    \notin   M_0$.     By
      Lemma~\ref{lem:v0matching},  either $vw_1$  is removable  in the
      brick $G_0$ or $G_0$ has an  edge that is removable and does not
      lie in $M_0  \cup \partial(X_0)$. In both cases, $G_1$  has a removable
      edge  that does  not lie  in $M_0  \cup C$,  a contradiction  to
      Corollary~\ref{cor:rem=>M0}.
 
      In all cases considered, we derived a contradiction. We deduce
      that $G_1$ is a brick. Likewise, a similar argument may be used
      to prove that $G_2$ is also a brick.
    \end{proof}

    \begin{lem}
      $C \subseteq M_0$.
    \end{lem}
    \begin{proof}
      Suppose, to  the contrary,  that $C  - M_0$  contains an
      edge, $e$.  By Corollary~\ref{cor:rem=>M0}, at least  one of the
      graphs  $G_1-e$ and  $G_2-e$  is not  matching covered.   Adjust
      notation so that  $G_1-e$ is not matching covered.  That is, the
      edge   $e$   is  not   removable   in   the  brick   $G_1$.   By
      Lemma~\ref{lem:v0matching},  $G_1$ has  a  removable edge,  $f$,
      that  does  not  lie  in   $M_0  \cup  C$.   Thus,  $G_1-f$  and
      $G_2-f=G_2$  are  both  matching  covered,  a  contradiction  to
      Corollary~\ref{cor:rem=>M0}. Indeed, $C \subseteq M_0$.
    \end{proof}

    \begin{lem}
      \label{lem:H-solid}
      If a $C$-contraction $H$ of $G$ is solid then $H=K_4$.
    \end{lem}
    \begin{proof}
      Adjust  notation  so  that  $G_1$  is  solid.   Assume,  to  the
      contrary, that $G_1 \ne K_4$. The cut $C$ consists only of three
      edges in $M_0$  and $G_1$ is a brick. Thus,  $G_1$ is simple but
      is not a wheel having $\overline{x}$  as a hub.  By the Lemma on
      Wheels, $G_1$ has a removable edge $e$ that does not lie in $M_0
      \cup C$. Thus, $G_1-e$ and $G_2-e=G_2$ are both matching covered
      and   $e    \notin   M_0$.   This   is    a   contradiction   to
      Corollary~\ref{cor:rem=>M0}. Indeed, $G_1=K_4$, as asserted.
    \end{proof}

    \begin{lem}
      \label{lem:H-nonsolid}
      If a $C$-contraction $H$ of $G$ is  not solid then $H$ is one of
      the four basic nonsolid bricks.
    \end{lem}
    \begin{proof}
      Suppose  that a  $C$-contraction of  $G$ is  not solid.   Adjust
      notation so that  $G_1$ is not solid. By induction,  $G_1$ has a
      conformal  minor $J$  that is  one  of the  four basic  nonsolid
      bricks.   Thus, some  bisubdivision $H$  of $J$  is a  conformal
      subgraph of $G_1$.  Assume that $G_1$  is not $J$. As $G_1$ is a
      brick,  if $G_1=H$  then $G_1=J$.  In this  case, the  assertion
      holds.

      We may  thus assume that $H$  is a proper subgraph  of $G_1$. By
      Theorem~\ref{thm:conformal}, $G_1$ has a  removable ear $R$ such
      that  $H$ is  a conformal  subgraph of  $G_1-R$. As  $G_1$ is  a
      brick,  it follows  that the  edges of  $R$ constitute  either a
      removable edge or a removable doubleton.~\footnote{It can be
	shown that $R$ is a singleton, but that is not necessary in
	this argument.} 

      If  $R$  and $C$  are  disjoint  then  $G_1-R$ is  matching  covered and
      nonsolid (by Corollary~\ref{cor:conformal-minor-solidity}),
      and  $G_2-R=G_2$  is  matching  covered.  This  is  a
      contradiction to  Lemma~\ref{lem:R-rem-both=>tight-solid}. Thus,
      $R$ contains an edge, $e$, in $C$. Clearly, $e$ is the only edge
      of $R$  in $C$.  Let  $S$ be a  minimal class of  the dependence
      relation  in $G_2$ induced by edge $e$.
As $G_2$  is a  brick, $G_2-S$  is matching
      covered.   If $e  \in S$  then $G_1-(R  \cup S)$  is $G_1-R$,  a
      nonsolid  matching  covered  graph,  and  $G_2-(R  \cup  S)$  is
      $G_2-S$, a matching covered  graph.  Alternatively, if $e \notin
      S$ then $G_1-S=G_1$ is nonsolid and $G_2-S$ is matching covered.
      In   both   alternatives,   we   derive   a   contradiction   to
      Lemma~\ref{lem:R-rem-both=>tight-solid}. We deduce that $G_1$ is
      one of the four basic nonsolid bricks.
    \end{proof}

    Let us denote the bicorn by $R_8$ and the tricorn by $R_{10}$.  We
    now   know  that   $G_1$   and   $G_2$  are   both   in  the   set
    $\{K_4,\overline{C_6}, R_8, R_{10}, \mathbb{P}\}$. We now begin by
    showing that  in fact  no $C$-contraction of  $G$ is  the Petersen
    graph and no $C$-contraction of $G$ is the tricorn.

    \begin{lem}
      \label{lem:H-not-P-not-R10}
      Neither $G_1$ nor $G_2$ is in $\{\mathbb{P},R_{10}\}$.
    \end{lem}
    \begin{proof}
      Assume, to the contrary, that $G_2$ is the Petersen graph. Every
      one of the  15 edges of $\mathbb{P}$ is removable  and 9 of them
      do     not    lie     in    $M_0 \cup C$,     a    contradiction     to
      Corollary~\ref{cor:rem=>M0}. As asserted, $G_2 \ne \mathbb{P}$.

      Suppose now,  to the  contrary, that $G_2$  is the  tricorn. The
      tricorn has  three removable edges, $e_i$,  $i=1,2,3$. The three
      removable edges of the tricorn lie in a hexagon, $H$, together with
      the edges $f_i$, $i=1,2,3$. See Figure~\ref{fig:tricorn}.

      \begin{figure}[!ht]
  	\centering
  	\psfrag{x}{$x$}
  	\psfrag{e1}{$e_1$}
  	\psfrag{e2}{$e_2$}
  	\psfrag{e3}{$e_3$}
  	\psfrag{f1}{$f_1$}
  	\psfrag{f2}{$f_2$}
  	\psfrag{f3}{$f_3$}
	\includegraphics [width=.4\textwidth] {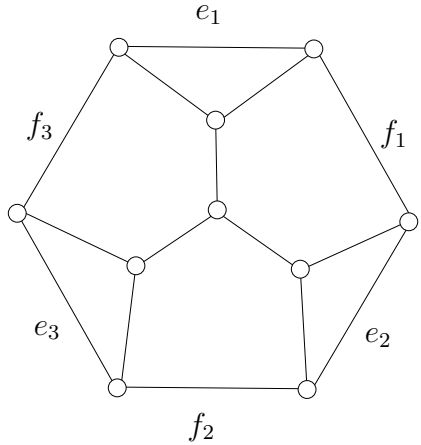}
  	\caption{The three removable edges of the tricorn: $e_1$,
  	  $e_2$ and $e_3$}
  	\label{fig:tricorn}
      \end{figure}

      If $M_0$ does not contain at least one of these three edges then
      again we get a  contradiction to Corollary~\ref{cor:rem=>M0}. We
      may  thus assume  that $\{e_1,e_2,e_3\}  \subset M_0$.   In this
      case, the contraction  vertex $x$ of $G_2$ cannot  be in $V(H)$,
      hence $H$  is an  $M_0$-alternating cycle.  We may  then replace
      $M_0$ by its symmetric difference with $E(H)$.  and again obtain
      a contradiction to Corollary~\ref{cor:rem=>M0}.

      We conclude  that $G_2 \notin \{R_{10},  \mathbb{P}\}$. The same
      conclusion holds for $G_1$.
    \end{proof}

    \begin{lem}
      \label{lem:one-solid}
      At least one $C$-contraction of $G$ is solid.
    \end{lem}
    \begin{proof}
      Assume   the  contrary.   By  Lemma~\ref{lem:H-nonsolid},   both
      $C$-contractions   of  $G$   are  basic   nonsolid  bricks. By
      Lemma~\ref{lem:H-not-P-not-R10}, $G_1$ and $G_2$ are both in
      $\{\overline{C_6},R_8\}$.

      Assume that $G_2=\overline{C_6}$. The brick $\overline{C_6}$ has
      three removable  doubletons $R_i:=\{e_i,f_i\}$,  $i=1,2,3$.  See
      Figure~\ref{fig:C6bar}.

     \begin{figure}[!ht]
  	\centering
  	\psfrag{e1}{$e_1$}
  	\psfrag{e2}{$e_2$}
  	\psfrag{e3}{$e_3$}
  	\psfrag{f3}{$f_3$}
  	\psfrag{f1}{$f_1$}
  	\psfrag{f2}{$f_2$}
  	\includegraphics [width=.4\textwidth] {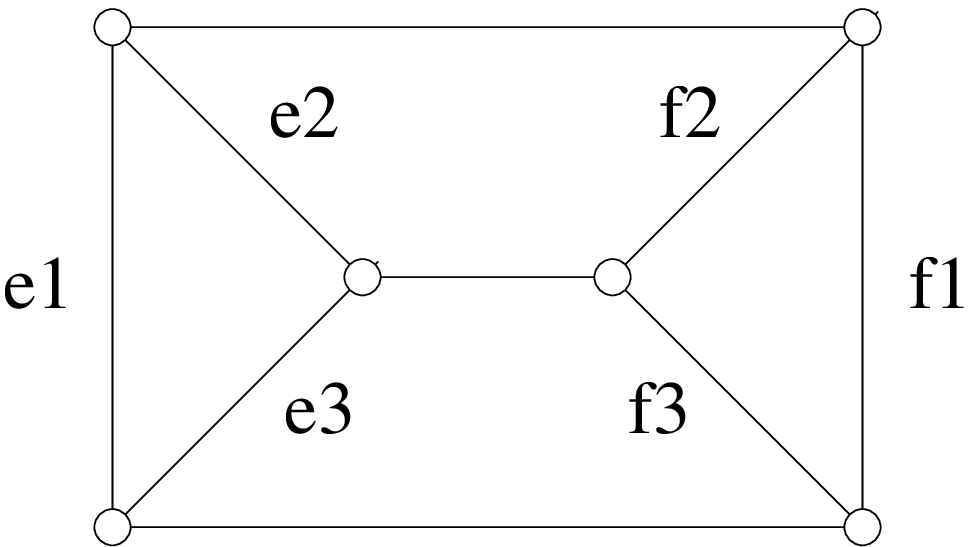}
  	\caption{The three removable doubletons of $\overline{C_6}$:
  	  $\{e_i,f_i\}$, $i=1,2,3$}
  	\label{fig:C6bar}
      \end{figure}

     No vertex of $\overline{C_6}$ is  incident to edges of all three
      doubletons.  Thus, $G_2$ has  a removable doubleton $R$ disjoint
      with $C$, hence $G_1-R=G_1$ is  nonsolid and $G_2-R$ is matching
      covered.        This      is       a      contradiction       to
      Lemma~\ref{lem:R-rem-both=>tight-solid}.

      Alternatively, assume that $G_2=R_8$.  The graph $R_8$ has three
      removable classes, the two doubletons $\{e_i,f_i\}$, $i=1,2$ and
      the  edge  $e$.  See   Figure~\ref{fig:bicorn}.

      \begin{figure}[!ht]
  	\centering
  	\psfrag{x}{$x$}
  	\psfrag{e1}{$e_1$}
  	\psfrag{e2}{$e_2$}
  	\psfrag{e}{$e$}
  	\psfrag{f1}{$f_1$}
  	\psfrag{f2}{$f_2$}
  	\psfrag{Q}{$Q$}
  	\includegraphics [width=.4\textwidth] {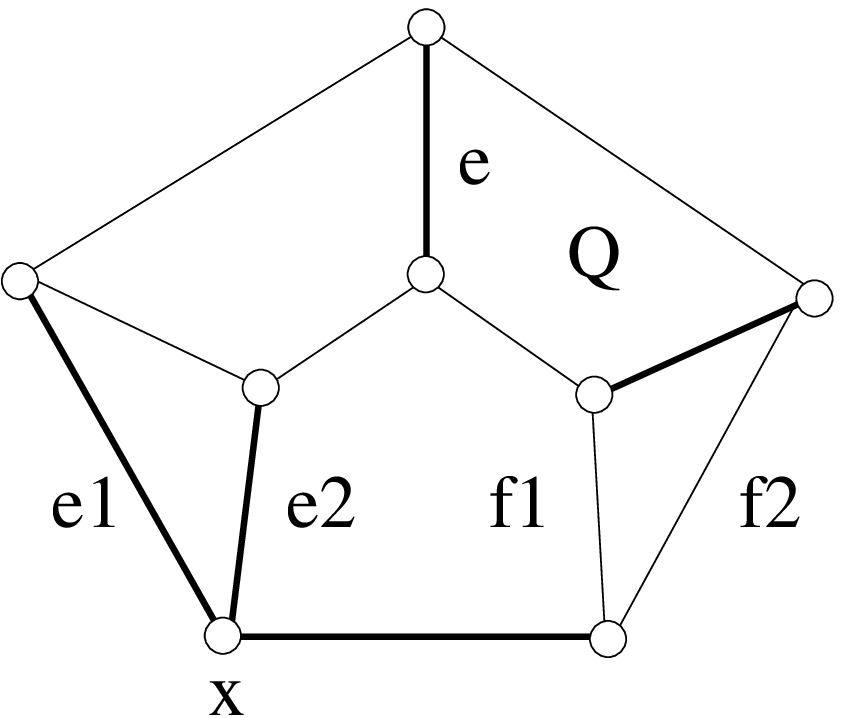}
  	\caption{The three removable classes of $R_8$:
  	  the doubletons $\{e_i,f_i\}$, $i=1,2$ and the edge $e$}
  	\label{fig:bicorn}
      \end{figure}

      The edge $e$  and $\{e_1,f_1\}$ are disjoint. Thus,  $G_2$ has a
      removable  class $R$  disjoint  with $C$,  hence $G_1-R=G_1$  is
      nonsolid   and  $G_2-R$   is  matching   covered.   This   is  a
      contradiction to Lemma~\ref{lem:R-rem-both=>tight-solid}.

      In all cases considered, we derived a contradiction. Indeed, at
      least one $C$-contraction of $G$ is solid.
    \end{proof}

    \bigskip
    If    $G_1$    and    $G_2$    are    both    solid    then,    by
    Lemma~\ref{lem:H-solid}, $G_1$ and $G_2$ are both $K_4$, therefore
    $G$ is $\overline{C_6}$. The assertion  holds in this case. We may
    thus assume  that at  least one  of $G_1$  and $G_2$  is nonsolid.
    Adjust    notation    so    that   $G_2$    is    nonsolid.     By
    Lemma~\ref{lem:one-solid},   the  brick   $G_1$   is  solid.    By
    Lemmas~\ref{lem:H-solid},                     \ref{lem:H-nonsolid}
    and~\ref{lem:H-not-P-not-R10}, the  brick $G_1$  is $K_4$  and the
    brick $G_2$ is either $\overline{C_6}$ or the bicorn.

    If $G_2$ is $\overline{C_6}$ then $G$ is the bicorn. Consider next
    the case in which $G_2$ is the bicorn. If $x$ is not incident with
    the  only  removable edge  $e$  of  $G_2$  but  $e \in  M_0$  (see
    Figure~\ref{fig:bicorn}), then  there is only one  possibility, up
    to  automorphisms.   The edge  $e$  lies  in an  $M_0$-alternating
    quadrilateral  $Q$  and we  may  replace  $M_0$ by  its  symmetric
    difference      with     $E(Q)$,      in     contradiction      to
    Corollary~\ref{cor:rem=>M0}. Thus, $x$  is an end of  $e$. In that
    case, $G$ is the tricorn.

    Indeed,  if $G$  is  not the  Petersen graph  then  $G$ is  either
    $\overline{C_6}$,  the  bicorn  or   the  tricorn.  The  proof  of
    Theorem~\ref{thm:four-nonsolid} is complete.
  \end{proof}

  \section{Equivalence of Problems~\ref{prb:nonsolid} and
    \ref{prb:prism-based}}
  \subsection{Thin and strictly thin edges}
  Motivated by the  problem of recursively generating  bricks, we were
  led to the notion of thin edges. An  edge $e$ of a brick $G$ is {\em
    thin} if the retract of $G-e$ is also a brick.  (Our definition of
  a  thin edge  in  \cite{clm06}  was phrased  in  terms  of sizes  of
  barriers,  but  is  equivalent  to   the  one  given  here.)   Using
  Theorem~\ref{thm:b-invariant-brick}, we  proved in  \cite{clm06} the
  following assertion.

  \begin{thm}\label{thm:thin-brick}
    Every brick distinct from $K_4$, $\overline{C_6}$ and $\mathbb{P}$
    (the Petersen graph) has a thin edge. \qed
  \end{thm}

  The  above theorem  implies the  following corollary  which has  the
  flavour of Theorem~\ref{thm:ear-decomp}.

  \begin{cor}\label{cor:thin-brick}
    Given any brick $G$, there exists a sequence $(G_1,G_2,\dots,G_r)$
    of bricks such that:
    \begin{enumerate}
    \item[(i)] $G_r=G$ and $G_1 \in \{K_4, \overline{C_6}, \mathbb{P} \}$; and
    \item[(ii)] for $1<i\leq r$, the brick $G_i$ has a thin edge $e_i$
      such that $G_{i-1}$ is the retract of $G_i-e_i$.
    \end{enumerate}
  \end{cor}

  This corollary is the basis  of a recursive procedure for generating
  bricks described  in \cite{clm06}. We  showed that there  exist four
  elementary `expansion  operations' which  can be  used to  build any
  brick   starting   from   one  of   $K_4$,   $\overline{C_6}$,   and
  $\mathbb{P}$.  (The simplest  of these  operations consists  of just
  adding  an edge  joining  two  distinct vertices  of  a given  brick
  $G$.  The  other  three  involve bi-splitting  vertices  and  adding
  edges.)

  \smallskip

  We  associate with  each thin  edge a  number called  its index,  as
  defined below.   Let $G$ be a  brick and let  $e$ be a thin  edge of
  $G$.  Then the retract of $G-e$, by definition, is a brick. The {\em
    index} of $e$ is:
  \begin{itemize}
  \item
    {\em zero}, if both ends of $e$ have degree four or more in $G$;
  \item
    {\em one}, if exactly one end of $e$ has degree three in $G$;
  \item
    {\em two}, if both  ends of $e$ have degree three  in $G$ and edge
    $e$ does not lie in a triangle;
  \item
    {\em three}, if both ends of $e$ have degree three in $G$ and edge
    $e$ lies in a triangle.
  \end{itemize}

  Examples of thin edges of indices  one, two, and three are indicated
  by  solid  lines  in  the   three  bricks,  respectively,  shown  in
  Figure~\ref{fig:bricks-indices}.

  \begin{figure}[htb]
    \begin{center}
      \psfrag{a}{$(a)$}
      \psfrag{b}{$(b)$}
      \psfrag{c}{$(c)$}
      \psfrag{e}{$e$}
      \psfrag{f}{$f$}
      \epsfig{file=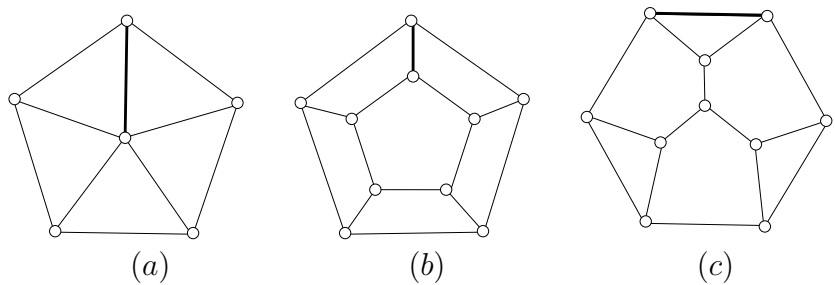, width=.8\textwidth}
    \end{center}
    \caption{(a) The  wheel $W_5$; (b)  the pentagonal prism;  (c) the
      tricorn}
    \label{fig:bricks-indices}
  \end{figure}

The following consequence of Theorem~\ref{thm:km-cubic} will
be useful later.

\begin{prp}\label{prp:inheritance}(\cite{komu16})
Let $G$ be a brick and $e$ be a thin edge of $G$.
For any cubic brick $J$, if the retract of $G-e$ is $J$-based then $G$
is also $J$-based.~\qed
\end{prp}

  \smallskip

  In  order to  establish recursive  procedures for  generating simple
  bricks, one needs the notion of a strictly thin edge. An edge $e$ of
  a simple brick $G$ is {\em strictly thin} if $e$ is thin and the retract of
  $G-e$ is  simple.  There are  five infinite families of bricks that
are free of strictly thin edges; these are
  (i)  odd wheels,  (ii)  prisms of  order $2 (\mbox{modulo }4)$,  (iii)
  M\"obius ladders of order  $0 (\mbox{modulo }4)$, (iv) staircases, and
  (v) truncated  biwheels. We refer  to bricks in these  five families
  together   with   the   Petersen   graph   as   {\em   Norine-Thomas
    bricks}. (Note  that $K_4$ is  the M\"obius ladder of  order four,
  and $\overline{C_6}$  is the  prism of order  six.) For  brevity, we
  shall denote the family of all Norine-Thomas bricks by ${\cal NT}$.

  \smallskip

  Norine  and  Thomas  established   the  following  strengthening  of
  Theorem~\ref{thm:thin-brick}.
  \begin{thm}\label{thm:strictly-thin-brick}(\cite{noth07})
    Every simple  brick $G$ which is  not a Norine-Thomas brick  has a
    strictly thin edge.\qed
  \end{thm}
  The work  of Norine and Thomas  is independent of our  work and uses
  entirely different  methods.  After learning about  the statement of
  their result,  we were able  to show that  it is possible  to derive
  Theorem~\ref{thm:strictly-thin-brick}            from            our
  Theorem~\ref{thm:thin-brick}.   Our  proof  appears  in an unpublished
  report \cite{clm08}. As an immediate consequence of the above theorem, we have:

  \begin{cor}\label{cor:strictly-thin-brick}
    Given any  simple brick  $G$,
    there  exists a  sequence $$(G_1,G_2,\dots,G_r)$$  of simple  bricks
    such that:
    \begin{enumerate}
    \item[(i)] $G_r=G$ and $G_1$ is in ${\cal NT}$; and
    \item[(ii)] for $1<i\leq  r$, the brick $G_i$ has  a strictly thin
      edge $e_i$ such that $G_{i-1}$ is the retract of $G_i-e_i$.\qed
    \end{enumerate}
  \end{cor}
  \smallskip

  In the same paper \cite{noth07}, Norine and  Thomas have also proved the following
  powerful generalization of Theorem~\ref{thm:strictly-thin-brick}; it
  belongs to a  class of theorems in structural graph  theory known as
  `splitter  theorems'.   To state  this  generalization,  we need  to
  define a  new class of  graphs.  The graph $T_{2k}^+$ is obtained from
  the truncated  biwheel $T_{2k}$ by  joining its hubs.   The extended
  Norine-Thomas family ${\cal  NT}^+$ is the union of  ${\cal NT}$ and
  $\{T^+_{2k}:~k\in\zz,k\geq 3\}$.

  \begin{thm}\label{thm:splitter-brick}
    Let $G$ be a  simple brick which is not in ${\cal NT}^+$
and let $J$ be a simple brick that is distinct from $K_4$ and $\overline{C_6}$.
If $J$ is a matching minor of $G$ then
there there  exists a sequence  $G_1,G_2,\dots,G_r$ of
    simple bricks such that:
    \begin{enumerate}
    \item[(i)] $G_r=G$ and $G_1=J$, and
    \item[(ii)] for $1 <  i \leq r$, the brick $G_i$ has  a strictly thin
      edge $e_i$ such that $G_{i-1}$ is the retract of $G_i-e_i$.
    \end{enumerate}
  \end{thm}

  Since any cubic  brick which is a  conformal minor of $G$  is also a
  matching minor  of $G$, the above  theorem is applicable to the case in
  which $J$ is a cubic brick, distinct from $K_4$ and $\overline{C_6}$, that happens to be
  conformal minor of $G$.

  \subsection{Proof of the equivalence}

Let us recall that every solid brick is $\overline{C_6}$-free. The only simple planar
solid bricks are the odd wheels. The only simple planar $\overline{C_6}$-free bricks
are staircases of order 0 (modulo 4), the tricorn, and of course the odd wheels.
We now show that apart from the staircases of order 0 (modulo 4), the tricorn,
and the ubiquitous Petersen graph, a simple brick is solid if and only if it is
$\overline{C_6}$-free. It suffices to prove the following:

  \begin{thm}\label{thm:equivalence}
    Any simple nonplanar nonsolid brick $G$, distinct from the Petersen
    graph, is $\overline{C_6}$-based.
  \end{thm}

\begin{proof} 
    As the only nonplanar members of  the family ${\cal NT}^+$ are the
    M\"obius  ladders of  order $0 (\mbox{modulo }4)$  and the  Petersen
    graph,   it    follows   that   $G\not\in   {\cal    NT}^+$.    Also, by
    Theorem~\ref{thm:four-nonsolid},  $G$  contains  one of  the  four
    basic  nonsolid bricks  as  a conformal  minor.   To complete  the
    proof, it suffices to show that  if $G$ has either the bicorn, the
    tricorn, or the Petersen graph as  a conformal minor, then it also
    has $\overline{C_6}$ as a conformal  minor.  Towards this end, let
    $J$    be    one    of    the    above-mentioned    bricks (that is, bicorn,
tricorn or the Petersen graph) such that $J$ is a conformal minor of $G$.     
    
    By
    Theorem~\ref{thm:splitter-brick},   there    exists   a   sequence
    $G_1,G_2,...,G_r$  of  simple bricks  such  that  (i) $G_r=G$  and
    $G_1=J$, and (ii) for $1 < i \leq r$,  the brick $G_i$ has a strictly
    thin  edge   $e_i$  such   that  $G_{i-1}$   is  the   retract  of
    $G_i-e_i$.   In   particular,   $G_1=J$    is   the   retract   of
    $G_{2}-e_{2}$. Since  $J$ is  a cubic  brick, it  follows that
    $e_{2}$ is a  strictly thin edge of index zero  and, hence, that
    $J=G_{2}-e_{2}$. In other words, $G_{2}$ is obtained from $J$
    by joining  two nonadjacent vertices  by an edge.
    By Proposition~\ref{prp:inheritance}, every cubic brick that is a conformal
    minor of $G_2$ is also a conformal minor of $G$. Thus, in order to  complete the
    proof, all we  need to do is to show that any brick obtained from
    either the bicorn,  the tricorn, or the Petersen  graph contains a
    bi-subdivision of $\overline{C_6}$ as  a conformal subgraph.  This
    is routine.
  \end{proof}

  \bibliographystyle{abbrv}

  \bibliography{clkm17}

\begin{thebibliography}{10}

\bibitem{bomu08}
J.~A. Bondy and U.~S.~R. Murty.
\newblock {\em Graph Theory}.
\newblock Springer, 2008.

\bibitem{calu00}
C.~N. Campos and C.~L. Lucchesi.
\newblock On the relation between the {P}etersen graph and the characteristic
  of separating cuts in matching covered graphs.
\newblock Technical Report~22, Institute of Computing, University of Campinas,
  Brazil, 2000.

\bibitem{clm99}
M.~H. Carvalho, C.~L. Lucchesi, and U.~S.~R. Murty.
\newblock Ear decompositions of matching covered graphs.
\newblock {\em Combinatorica}, 19:151--174, 1999.

\bibitem{clm02}
M.~H. Carvalho, C.~L. Lucchesi, and U.~S.~R. Murty.
\newblock On a conjecture of {L}ovász concerning bricks. {I}. {T}he
  characteristic of a matching covered graph.
\newblock {\em J.~Combin.~Theory Ser.~B}, 85:94--136, 2002.

\bibitem{clm02a}
M.~H. Carvalho, C.~L. Lucchesi, and U.~S.~R. Murty.
\newblock On a conjecture of {L}ovász concerning bricks. {I}{I}. {B}ricks of
  finite characteristic.
\newblock {\em J.~Combin.~Theory Ser.~B}, 85:137--180, 2002.

\bibitem{clm04}
M.~H. Carvalho, C.~L. Lucchesi, and U.~S.~R. Murty.
\newblock The perfect matching polytope and solid bricks.
\newblock {\em J.~Combin.~Theory Ser.~B}, 92:319--324, 2004.

\bibitem{clm05}
M.~H. Carvalho, C.~L. Lucchesi, and U.~S.~R. Murty.
\newblock Graphs with independent perfect matchings.
\newblock {\em J.~Graph Theory}, 48:19--50, 2005.

\bibitem{clm06}
M.~H. Carvalho, C.~L. Lucchesi, and U.~S.~R. Murty.
\newblock How to build a brick.
\newblock {\em Discrete Math.}, 306:2383--2410, 2006.

\bibitem{clm08}
M.~H. Carvalho, C.~L. Lucchesi, and U.~S.~R. Murty.
\newblock Generating simple bricks and braces.
\newblock Technical Report IC-08-16, Institute of Computing, University of
  Campinas, July 2008.

\bibitem{clm12}
M.~H. Carvalho, C.~L. Lucchesi, and U.~S.~R. Murty.
\newblock A generalization of {L}ittle's theorem on {P}faffian graphs.
\newblock {\em J.~Combin.~Theory Ser.~B}, 102:1241--1266, 2012.

\bibitem{elp82}
J.~Edmonds, L.~Lovász, and W.~R. Pulleyblank.
\newblock Brick decomposition and the matching rank of graphs.
\newblock {\em Combinatorica}, 2:247--274, 1982.

\bibitem{kaoz13}
K.~Kawarabayashi and K.~Ozeki.
\newblock A simpler proof for the two disjoint odd cycles theorem.
\newblock {\em J.~Combin.~Theory Ser.~B}, 103:313--319, 2013.

\bibitem{komu16}
N.~Kothari and U.~S.~R. Murty.
\newblock ${K}_{4}$-free and $\overline{C_6}$-free planar matching covered
  graphs.
\newblock {\em J.~Graph Theory}, 82:5--32, 2016.

\bibitem{lova83}
L.~Lovász.
\newblock Ear decompositions of matching-covered graphs.
\newblock {\em Combinatorica}, 3:105--117, 1983.

\bibitem{lova87}
L.~Lovász.
\newblock Matching structure and the matching lattice.
\newblock {\em J.~Combin.~Theory Ser.~B}, 43:187--222, 1987.

\bibitem{lopl86}
L.~Lovász and M.~D. Plummer.
\newblock {\em Matching Theory}.
\newblock Number~29 in Annals of Discrete Mathematics. Elsevier Science, 1986.

\bibitem{mccu01}
W.~McCuaig.
\newblock Brace generation.
\newblock {\em J.~Graph Theory}, 38:124--169, 2001.

\bibitem{noth07}
S.~Norine and R.~Thomas.
\newblock Generating bricks.
\newblock {\em J.~Combin.~Theory Ser.~B}, 97:769--817, 2007.

\bibitem{tutt47}
W.~T. Tutte.
\newblock The factorization of linear graphs.
\newblock {\em J.~London Math.~Soc.}, 22:107--111, 1947.

\end{thebibliography}

\end{document}